\documentclass[final]{siamltex}
\usepackage{amssymb, amsmath}
\usepackage{float,epsfig}
\usepackage{algpseudocode}
\usepackage{algorithm}

\newtheorem{exam}[theorem]{\bf Example}

\newcommand{\ba}{\begin{array}}
\newcommand{\ea}{\end{array}}

\newcommand{\be}{\begin{equation}}
\newcommand{\ee}{\end{equation}}
\newcommand{\beano}{\begin{eqnarray*}}
\newcommand{\eeano}{\end{eqnarray*}}

\def\C{{\mathbb C}}

\def \L{{\mathbb L}}

\def\lam{\lambda}
\def\sig{\sigma}

\def\diag{\mathrm{diag}}

\def\rank{\mathrm{rank}}

\def\rar{\rightarrow}

\def \nrank{\mathrm{nrk}}
\def \ind{\mathrm{Ind}}
\def \sp{\mathrm{Sp}}
\def \poles{\mathrm{Poles}}

\title{Linearizations for Rosenbrock system polynomials and rational matrix functions}

\author{Rafikul Alam \thanks{Department of Mathematics, IIT Guwahati, Guwahati - 781039, India ({\tt rafik@iitg.ernet.in, rafikul68@gmail.com }) Fax: +91-361-2690762/2582649.} \and Namita Behera \thanks{Department of Mathematics, IIT Guwahati,
Guwahati-781039, India,({\tt niku.namita@gmail.com}). }}
\begin{document}

\maketitle

\begin{abstract}
Our aim in this paper is two-fold: First, for computing zeros of a linear time-invariant (LTI) system $\Sigma$ in {\em state-space form}, we introduce a {\em ``trimmed structured  linearization"}, which we refer to as {\em Rosenbrock linearization}, of the Rosenbrock system polynomial $\mathcal{S}(\lam)$ associated with $\Sigma.$ We also introduce Fiedler-like matrices for $\mathcal{S}(\lam)$ and describe constructions of Fiedler-like pencils for $\mathcal{S}(\lam).$  We show that the Fiedler-like pencils of $\mathcal{S}(\lam)$ are Rosenbrock linearizations of the system polynomial $\mathcal{S}(\lam).$ Second, with a view to developing a direct method for solving rational eigenproblems, we introduce {\em``linearization"} of a rational matrix function. We describe a state-space framework for converting a rational matrix function $G(\lam)$ to an {\em ``equivalent"} matrix pencil $\mathbb{L}(\lam)$ of smallest dimension such that $G(\lam)$ and $\mathbb{L}(\lam)$ have the same {\em ``eigenstructure"} and we refer to such a pencil $\mathbb{L}(\lam)$ as a ``linearization" of $G(\lam).$ Indeed, by treating $G(\lam)$ as the transfer function of an LTI system $\Sigma_G$ in state-space form via state-space realization, we show that the Fiedler-like pencils of the Rosenbrock system polynomial associated with $\Sigma_G$ are  ``linearizations" of $G(\lam)$ when the system $\Sigma_G$ is both controllable and observable.

\end{abstract}



\section{Introduction} The dynamical behavior of a linear system is often governed by a set of linear differential and algebraic equations.
 We consider a linear time-invariant (LTI) system $\Sigma$ in {\em state-space form }(SSF) given by~\cite{rosenbrock70} \begin{equation} \label{rassf}
\begin{array}{ll} \Sigma:  &
\begin{array}{ll}
 E \dot{x}(t) = A x(t) + Bu(t) \\
 y(t) = C x(t) + P(\frac{d}{dt}) u(t),
\end{array}
\end{array}
 \end{equation}
where $A$ and $E$ are $r$-by-$r$ matrices with $E$ being nonsingular, $P(\lambda)$ is an $n$-by-$n$ matrix polynomial and  $ B, C $ are matrices of appropriate dimensions. The system $\Sigma$ in SSF is a prototype of general LTI systems whose dynamical behaviors are governed by matrix polynomials.  Indeed, it is shown by Rosenbrock~\cite{rosenbrock70} that a  general LTI system whose dynamical behavior is governed by a set of equations given by $$ T(\delta_t) x(t) = U(\delta_t)u(t) \mbox{ and } y(t) = V(\delta_t) x(t)+ W(\delta_t) u(t)$$ can be reduced to an LTI system in SSF such as $\Sigma$ by means of {\em strict system equivalence}, where  $T(\lam), U(\lam), V(\lam), W(\lam)$ are matrix polynomials of appropriate dimensions with $T(\lam)$ being  regular and $\delta_t :=\frac{d}{dt}$ is time-derivative operator. We, therefore, consider LTI systems only in state-space forms.

The {\em transfer function}~\cite{rosenbrock70}  of the LTI system $\Sigma$ is a rational matrix function $G(\lam)$ given by
$G(\lam) := P(\lam) + C(\lam E - A)^{-1}B.$  The  {\em zeros} of the transfer function $G(\lam)$ are called the {\em transmission zeros} of the LTI system $\Sigma$ and are defined as follows~\cite{APMA06, rosenbrock70}. If $ \diag(\phi_1(\lam)/\psi_1(\lam), \cdots, \phi_k(\lam)/\psi_k(\lam), 0, \cdots, 0)$ is the Smith-McMillan form of $G(\lam)$ then the zeros of the  polynomials $\phi_1(\lam),\ldots, \phi_k(\lam)$ are called the  {\em zeros} of the rational matrix function $G(\lam)$~\cite{rosenbrock70, vardulakis, APMA06}.

The {\em Rosenbrock system polynomial}  (also referred to as the Rosenbrock system matrix)  associated with the LTI system $\Sigma$ is an $(n+r)\times (n+r)$ matrix polynomial  $\mathcal{S}(\lam)$ given by~\cite{rosenbrock70, vardulakis}
\be \label{sg} \mathcal{S}(\lam) := \left[
                        \begin{array}{c|c}
                          P(\lam) & C \\
                          \hline
                          B & (A - \lam E) \\
                        \end{array}
                      \right]. \ee   Note that the LTI system $\Sigma$ can be rewritten as $ \mathcal{S}(\delta_t ) \left[ \begin{array}{c} u(t) \\ x(t) \end{array}\right] = \left[ \begin{array}{c} y(t) \\ 0 \end{array}\right].$  For an ansatz $\left[\ba{c} u \\ x \ea\right] e^{\lam t}$, we have  $\mathcal{S}(\lam) \left[ \begin{array}{c} u \\ x \end{array}\right] e^{\lam t} = \left[ \begin{array}{c} y(t) \\ 0 \end{array}\right].$
 Hence the ansatz  $\left[\ba{c} u \\ x \ea\right] e^{\lam t}$ produces zero output in the system $\Sigma$ when $\lam$ is an eigenvalue of $\mathcal{S}(\lam)$ and $\left[ \ba{c} u\\ x \ea \right]$ is a corresponding eigenvector.  The eigenvalues of $\mathcal{S}(\lam)$ are called {\em invariant zeros} of the LTI system $\Sigma$~\cite{APMA06, rosenbrock70}. Various zeros such as transmission zeros, invariant zeros  and { \em decoupling zeros} of LTI systems  play an important role in the analysis of Linear Systems~~\cite{rosenbrock70, vardulakis, APMA06}.

Computation of zeros of an LTI system is a major task in Linear Systems Theory.  For the special case when $P(\lam)$ in (\ref{rassf}) is a constant matrix polynomial, that is, when $P(\lam) := D$, the transmission zeros of $\Sigma$  can be computed, for example, by {\sc matlab} software command {\tt tzero}. However, for a general LTI system in SSF such as the LTI system $\Sigma,$  computation of zeros of the LTI system is a challenging task. We, therefore, consider the following problem.

\vspace{1ex}
{\sc Problem-I.}  Compute  zeros (transmission zeros, invariant zeros, and decoupling zeros) of an LTI system $\Sigma$ in SSF.
\vspace{1ex}

We develop a direct method for computing zeros especially transmission zeros and invariant zeros of an LTI system in SSF.
An obvious direct method for computing invariant zeros of the LTI system $\Sigma$ is to consider a linearization~\cite{gohberg82, mackey2006vs} of the Rosenbrock system polynomial $\mathcal{S}(\lam)$  and then solve the resulting generalized eigenvalue problem. However, such a linearization based method has serious drawbacks. Firstly, if the degree of the polynomial $P(\lam)$ is $m$ then a linearization of $\mathcal{S}(\lam)$  would result in an $m(n+r)\times m(n+r)$ matrix pencil $\L(\lam)$ which would be too large for a modest value of $m$ when $n$ and $r$ are large. Secondly, such a linearization $\L(\lam)$ would have  (spurious) eigenvalues at $\infty$  when $m>1$ even when the leading coefficient of $P(\lam)$ is nonsingular. It is well known that  infinite eigenvalues of a pencil often pose numerical difficulty in computing eigenvalues of the pencil accurately. Thus computation of eigenvalues of the pencil $\L(\lam) $ with multiple eigenvalues at infinity would be inadvisable. Thirdly, although the Rosenbrock system polynomial $\mathcal{S}(\lam)$ is a structured matrix polynomial, a linearization  $\L(\lam)$ of $\mathcal{S}(\lam)$ may not preserve the structure, that is, the pencil $\L(\lam)$ may not be the Rosenbrock system matrix of an LTI system in SSF.

To overcome the drawbacks just outlined, we introduce a trimmed structured linearization of the Rosenbrock system polynomial $\mathcal{S}(\lam)$ which we refer to as  {\em Rosenbrock linearization } of $\mathcal{S}(\lam).$ More precisely, we introduce Rosenbrock linearization of $\mathcal{S}(\lam)$ which is an $(mn+r) \times (mn+r)$ matrix pencil $\mathbb{T}(\lam)$  such that $\mathbb{T}(\lam)$ is the  Rosenbrock system matrix of an appropriate LTI system in SSF and that $\mathbb{T}(\lam)$ is equivalent to $\mathcal{S}(\lam)$ in the sense of {\em system equivalence}. We also introduce Fiedler-like pencils for the system polynomial $\mathcal{S}(\lam)$ and show that the Fiedler-like pencils are Rosenbrock linearizations of $\mathcal{S}(\lam).$ We show that a Rosenbrock linearization of $\mathcal{S}(\lam)$ deflates infinite eigenvalues of $\mathcal{S}(\lam)$ when the leading coefficient of $P(\lam)$ is nonsingular.

Analysis  of transmission zeros of LTI systems naturally leads to zeros of rational matrix functions~\cite{rosenbrock70, vardulakis, APMA06, kailath}.
On the other hand, rational eigenvalue problems arise in many applications such as in acoustic emissions of high speed trains, calculations of quantum dots, free vibration of plates with elastically attached masses, and vibrations of fluid-solid structures, see~\cite{volker2004, solo, hwang04, vossh06, vmct06, planchard82, planchard89, voss1, voss82} and references therein. For example, the rational eigenvalue problem (REP)
$$G(\lam)x := -Ax + \lambda Bx + \lam^{2}\sum \limits _{j=1}^{k} \frac{1}{\omega_{j}-\lambda}C_{j}x =0, $$ arises when a generalized linear eigenproblem is condensed exactly~\cite{volker2004}, and the REP
$$
G(\lam)x := (K - \lambda M + \sum \limits _{j=1}^{k} \frac{\lambda}{\lambda-\sigma_{j}}C_{j})x =0,
$$
arises in the study of the vibrations of fluid solid structures,
where $K=K^{T}$ and $M=M^{T}$ are positive definite and $C_{j}=C_{j}^{T}, j= 1:k,$ have low ranks~\cite{planchard82, planchard89, voss1}.

Thus computation of eigenvalues and eigenvectors of a rational matrix function is an important task which requires development of efficient numerical methods. The development of numerical method for solving a rational eigenvalue problem  is an emerging area of research~\cite{volker2004, voss05, ruhe73, hvoss04, voss04, vossh06, solo}. We therefore consider the following problem.

\vspace{1ex}
{\sc Problem-II.}   Develop a direct method for solving an REP $G(\lam)u =0,$ where $G(\lam)$ is a rational matrix function.
\vspace{1ex}

An obvious direct method to solve an REP  is to transform the REP to a polynomial eigenvalue problem (PEP) by clearing out the denominators in the rational matrix function followed by linearization of the resulting PEP to obtain a generalized eigenvalue problem (GEP). Schematically,
$$  \mbox{ REP } \longrightarrow\mbox{ PEP } \longrightarrow \mbox{ GEP }  \longrightarrow \mbox{ SOLUTION}. $$
A downside of this brute-force {\em ``polynomialization"} of an REP is that the transformation from REP to PEP may introduce spurious eigenvalues, which may be difficult to detect and remove. Moreover, the transformation from REP to PEP followed by linearization may result in a GEP of very large  dimension especially when the rational matrix function has a large number of poles~\cite{bai11}. For an illustration  of spurious eigenvalues, consider the rational matrix function \be\label{exnoevl}  G(\lam) := \left[ \begin{matrix}
1 & \frac{1}{\lam-2} \\ 0 & 1
\end{matrix}\right].\ee
Note that $G(\lam)$ does not have an eigenvalue ($\infty$ is not an eigenvalue of $G(\lam)$) whereas its polynomialization $P(\lam) := (\lam-2) G(\lam)$ has a multiple eigenvalue at $\lam :=2.$

On the other hand, nonlinear eigensolvers such as Newton method  and nonlinear Rayleigh-Ritz methods (e.g., nonlinear Arnoldi, rational Krylov, Jacobi-Davidson) may be suitable when a few eigenpairs are desired  but their convergence analysis is a challenging task, see~\cite{ruhe73, hvoss04, voss05, voss04, vossh06, solo} and references therein.

Thus a better alternative would be to ``linearize" an REP to obtain a GEP of least dimension in such a way that avoids a polynomialization of the REP.
Such a method based on {\em minimal realization} of a rational matrix function is proposed in~\cite{bai11} which avoids polynomialization of REP. Indeed, by considering a minimal realization of $G(\lam)$ of the form \be \label{rlg}  G(\lam) = \sum_{j = 0}^{m} \lambda^{j}A_{j} +C(\lam E - A)^{-1}B, \ee
where $A, E, C, B $ are constant matrices of appropriate dimensions with $E$  being nonsingular,  it is shown in~\cite{bai11} that the eigenvalues and eigenvectors of $G(\lam)$ can be computed by solving the generalized eigenvalue problem for the pencil
\begin{equation}
\label{compr} \mathcal{C}(\lam) := \lam \left[
                         \begin{array}{cccc|c}
                           A_{m} &  &  &  &  \\
                            & I_{n} &  &  &  \\
                            &  & \ddots &  &  \\
                            &  &  & I_{n} &  \\
                            \hline
                            &  &  &  & -E \\
                         \end{array}
                       \right] +
\left[ \begin{array}{cccc|c}
                  A_{m-1} & A_{m-2}& \cdots & A_{0}& C  \\
                  -I_{n} & 0 & \cdots & 0 &  \\
                   & \ddots &  & \vdots  & \\
                   & & -I_{n} &0& \\
                   \hline
                  &  &  & B & A \\
                \end{array}
              \right],
\end{equation} where the void entries represent zero entries.
The pencil $\mathcal{C}(\lam)$ is referred  to as  a {\em companion linearization} of $G(\lam)$ in~\cite{bai11} and   can be thought of as an extension of the companion linearization  \begin{equation}
\label{cpmpp} C_P(\lambda) := \lambda \left[
                                  \begin{array}{cccc}
                                    A_{m} & 0 & \cdots & 0 \\
                                    0 & I_{n} & \ddots & \vdots \\
                                    \vdots & \ddots & \ddots& 0 \\
                                    0 & \cdots & 0 & I_{n} \\
                                  \end{array}
                                \right] + \left[
                                            \begin{array}{cccc}
                                              A_{m-1} & A_{m-2} & \cdots & A_{0} \\
                                              -I_{n} & 0 & \cdots & 0 \\
                                              \vdots & \ddots & \ddots & \vdots \\
                                              0 & \cdots & -I_{n} & 0 \\
                                            \end{array}
                                          \right]
\end{equation}
of the matrix polynomial $P(\lam) := \sum_{j=0}^{m}\lam^{j}A_j $ to the case of $G(\lam).$  The pencil $C_P(\lam)$  is also  known as the first companion form of $P(\lam)$~\cite{gohberg82}. The size of the pencil $\mathcal{C}(\lam)$ is usually much smaller than that of a pencil obtained by polynomialization of $G(\lam)$ followed by linearization specially when the coefficient matrices of $G(\lam)$ have low ranks~\cite{bai11}.

The pencil $\mathcal{C}(\lam)$ has been referred to as ``companion linearization" of $G(\lam)$ in \cite{bai11}. So it is natural to  ask: {\em Is the pencil $\mathcal{C}(\lam)$ a linearization of $G(\lam)$? If yes, then in what sense does $\mathcal{C}(\lam)$ define a linearization of $G(\lam)?$ }  It is well known that the companion form $C_P(\lam)$ is a linearization  of $P(\lam)$ in the sense of unimodular equivalence~\cite{gohberg82, mackey2006vs}.

With a view to addressing {\sc Problem-II} as well as the questions raised above, we  consider yet another problem.

\vspace{1ex}
{\sc Problem-III.}  Let $ \diag(\phi_1(\lam)/\psi_1(\lam), \cdots, \phi_k(\lam)/\psi_k(\lam), 0, \cdots, 0)$ be the Smith-McMillan form of $G(\lam).$ Does there exist a matrix pencil $\mathbb{L}(\lam)$ of smallest dimension such that $\diag( I_p, \phi_1(\lam), \cdots, \phi_k(\lam), 0, \cdots, 0)$ is the  Smith form of $\mathbb{L}(\lam)$ for some appropriate $p?$ Describe a framework for construction of such a pencil when it exists.
\vspace{1ex}

The pencil $\mathbb{L}(\lam)$, when exists, can be considered as a {\em ``linearization"} of $G(\lam)$ as it preserves the spectrum (zeros) of $G(\lam)$ and the partial multiplicities of the eigenvalues. Observe that $\mathbb{L}(\lam)$ exists and is a linearization of $G(\lam)$ when $G(\lam)$ is a matrix polynomial. Consequently, the linearization in {\sc Problem-III}, when exists,  is an extension of linearization of a matrix polynomial to the case of a rational matrix function.

We provide a complete solution of {\sc Problem-III}. We show that computation of eigenvalues and eigenvectors of a rational matrix function is closely related to computation of transmission zeros of an LTI system.  We therefore embed an REP into an appropriate LTI system in SSF. More precisely,  we treat a rational matrix function $G(\lam)$ as the transfer function of an LTI system $\Sigma$ in SSF via state-space realization~\cite{rosenbrock70}. In such a case, the zeros of $G(\lam)$ become the transmission zeros of the LTI system $\Sigma.$ Thus, we compute transmission zeros of the LTI system $\Sigma$ from which we recover the eigenvalues of $G(\lam).$  We show that  a Rosenbrock linearization $\mathbb{L}(\lam)$ of the LTI system $\Sigma$ is a linearization of $G(\lam)$ in the sense described in {\sc Problem-III} when the LTI system $\Sigma$ is {\em controllable} as well as {\em observable}.  In particular, we show that the companion form $\mathcal{C}(\lam)$ in (\ref{compr}) is indeed a linearization of $G(\lam)$ in the sense described in {\sc Problem-III} when the realization of $G(\lam)$ in (\ref{rlg}) is minimal. Thus, schematically, our strategy for solution of {\sc Problem-II} can be summed as follows:
\vspace{2ex}
\begin{center}
Rational matrix function  $\longrightarrow$ State-space realization $\longrightarrow$ Linearization $\longrightarrow$ Solution.
\end{center}
 \vspace{2ex}

Finally, we mention that the set of transmission zeros of an LTI  system $\Sigma$ forms a subset of the set of invariant zeros of the LTI system and the  transmission zeros coincide with the invariant zeros when the LTI system $\Sigma$ is controllable as well as observable. Thus, {\sc Problem-II} via state-space realization is subsumed by {\sc Problem-I.} There are efficient methods for computing a (minimal) state-space realization of a rational matrix function~\cite{APMA06, kailath, rosenbrock70}. For example, the {\sc matlab}  command {\tt minreal} can be used to compute a minimal realization. We therefore focus on solving {\sc Problem-I} and address the issues in {\sc Problem-III} by assuming a minimal state-space realization.

The rest of the paper is organized as follows. Section~2 describes spectra and poles of rational matrix functions. Section~3 introduces Fiedler matrices associated with a Rosenbrock system matrix and describes construction of Fiedler pencils. Section~4  introduces Rosenbrock linearizations for Rosenbrock system matrix and shows that Fiedler pencils are indeed Rosenbrock linearizations for Rosenbrock system polynomials. Finally, Section~5 introduces linearization of a rational matrix function, which extends the notion of  linearization of a matrix polynomial to the case of a rational matrix function, and describes construction of such a linearization via minimal state-space realization of the rational matrix function. \\

{\bf Notation.} We denote by $\C[\lam]$ and $\C(\lam),$ respectively, the polynomial ring and the field of rational functions of the form $p(\lam)/{q(\lam)},$ where $p(\lam)$ and $q(\lam)$ are scalar polynomials in $\C[\lam].$  Further, we denote by $ \C^{m\times n}, \C[\lam]^{m\times n}$ and $ \C(\lam)^{m\times n},$ respectively, the vector spaces (over $\C$) of $m$-by-$n$ matrices, matrix polynomials and rational matrix functions. We denote the {\em Smith form}~\cite{rosenbrock70} of $ X(\lam) \in \C[\lam]^{m\times n}$ by $ {\mathbf{SF}}(X(\lam))$ and the {\em Smith-McMillan} form~\cite{rosenbrock70} of $ Y(\lam) \in \C(\lam)^{m\times n}$ by $ {\mathbf{SM}}(Y(\lam)).$  The normal rank~\cite{rosenbrock70} of $ X(\lam) \in \C(\lam)^{m\times n}$ is denoted by $ \nrank(X)$ and is given by $\nrank(X) := \max_{\lam}\rank(X(\lam))$ where the maximum is taken over all $\lam \in \C$ which are not poles of the entries of $X(\lam).$ We denote  the $j$-th column of the $n \times n$ identity matrix $I_n$ by $e_j$ and the transpose of a matrix $A$ by $A^T.$ We denote the Kronecker product of matrices $A$ and $B$ by $A \otimes B.$


\section{Eigenvalues, zeros and poles} Let $ P(\lam) \in \C[\lam]^{m\times n}.$ If $\nrank(X) = \min(m, n)$ then $P(\lam)$ is said to be regular, otherwise $P(\lam)$ is said to be singular. A matrix polynomial $U(\lam) \in \C[\lam]^{n\times n}$ is said to be unimodular if $\det(U(\lam)),$ the determinant of $U(\lam),$ is a nonzero constant for all $\lam \in \C.$  Suppose that the Smith form~\cite{rosenbrock70} of $P(\lam)$ is given by $\mathbf{SF}(P(\lam)) = \diag(\phi_1(\lam), \ldots, \phi_k(\lam), 0, \ldots, 0),$  where $\phi_1(\lam), \ldots, \phi_k(\lam)$ are invariant polynomials (unique monic polynomials such that $ \phi_i$ divides $\phi_{i+1}$ for $ i=1, \ldots, k-1$) of $P(\lam)$   and $ k= \nrank(P).$ We define the {\em zero polynomial} $\phi_P(\lam)$ of $P(\lam)$ and the {\em spectrum} (also called the zeros) $\sp(P)$ of $P(\lam)$  by
\be \phi_{P}(\lam) := \prod _{j=1}^{k} \phi_{j}(\lam)  \,\, \mbox{ and } \,\,  \sp(P) := \{ \lam \in \C : \phi_P(\lam) =0\}.\ee
A complex number $\lam \in \C$ is said to be an eigenvalue of $P(\lam)$ if  $\rank(P(\lam)) < \nrank(P).$ It follows that $\mu $ is an eigenvalue of $P(\lam)$ if and only if $ \phi_P(\mu) = 0.$ Thus the spectrum $\sp(P)$ consists of (finite) eigenvalues of $P(\lam).$ We mention that the spectrum $\sp(P)$ as defined above does not include infinite eigenvalues of $P(\lam).$

Now consider the LTI system $\Sigma$ given in (\ref{rassf}) and the associated  Rosenbrock system polynomial $\mathcal{S}(\lam)$ given in (\ref{sg}).  Then the spectrum  $\sp(\mathcal{S}) $ of $\mathcal{S}(\lam)$ is called the {\em invariant zeros} of the LTI system $\Sigma$. If $ \lam \in \sp(\mathcal{S})$ then   a corresponding eigenvector  $[ u^T, x^T]^T$ with $ u \in \C^n$ and $ x \in \C^r$  is called an {\em invariant zero direction}  and, the components $u$ and $x$ are, respectively,  called {\em input direction} and  {\em state  direction} of the LTI system $\Sigma.$ Further, the spectrum of the matrix pencils \be \label{dcouple}  M(\lam):= \left[\begin{matrix} A-\lam E  & B \end{matrix} \right] \,\,  \mbox{ and  } \, \, N(\lam) := \left[\begin{matrix} A-\lam E \\ C\end{matrix}\right]\ee are, respectively, called the {\em input decoupling zeros}  and  {\em output decoupling zeros} of the LTI system $\Sigma.$ We refer to~\cite{APMA06, rosenbrock70, vardulakis} for more on zeros of an LTI system.

Next, let $G(\lam)\in \mathbb{C}(\lam)^{m \times n}$ be a rational matrix function. Suppose that the Smith-McMillan form $\mathbf{SM}(G(\lam))$ of $G(\lam)$  is given by~\cite{rosenbrock70, vardulakis}
\be \label{smform}
\mathbf{SM}(G(\lam)) = \diag\left( \frac{\phi_{1}(\lam)}{\psi_{1}(\lam)}, \cdots, \frac{\phi_{k}(\lam)}{\psi_{k}(\lam)}, 0_{m-k, n-k}\right).\ee
where the scalar polynomials $\phi_{i}(\lam)$ and $ \psi_{i}(\lam)$ are monic, are pairwise coprime and,  $\phi_{i}(\lam)$ divides $\phi_{i+1}(\lam)$ and $\psi_{i+1}(\lam)$ divides $\psi_{i}(\lam),$ for $i= 1, 2, \ldots, k-1$.
The polynomials $\phi_1(\lam), \ldots, \phi_k(\lam)$ and $ \psi_1(\lam), \ldots, \psi_k(\lam)$ are called {\em invariant zero polynomials} and {\em invariant pole polynomials} of $G(\lam),$ respectively. Define the {\bf zero polynomial} $\phi_G(\lam)$ and the {\bf pole polynomial} $\psi_G(\lam)$ of $G(\lam)$  by
\begin{equation} \label{zppoly}
\phi_{G}(\lam) := \prod _{j=1}^{k} \phi_{j}(\lam) \,\,\, \mbox{ and } \,\,\, \psi_{G}(\lam) := \prod _{j=1}^{k} \psi_{j}(\lam).
\end{equation}

\begin{definition} [Zeros and poles \cite{rosenbrock70}]  Let $ G(\lam) \in \C(\lam)^{m\times n}.$ Let $\phi_G(\lam)$ and $\psi_G(\lam)$ be the zero and pole polynomials of $G(\lam),$ respectively. A complex number $ \lam $ is said to be a  zero of $G(\lam)$ if $ \phi_G(\lam) =0.$
A complex number $ \lam$ is said to be a  pole of $G(\lam)$ if $\psi_G(\lam) =0.$
\end{definition}

We mention that  $G(\lam)$ is said to have a zero (pole) at $\lam = \infty$  if  $G(1/\lam)$ has a zero (pole) at $\lam = 0.$ We, however, consider only finite zeros and poles of $G(\lam)$. It follows that a complex number $ \mu$ is a pole of $G(\lam)$ if and only if $ \|G(\lam)\|_2 \rar \infty$ as $ \lam \rar \mu.$ We denote the set of poles of $G(\lam)$ by $\poles(G)$ and define the {\bf spectrum} $\sp(G)$ of $G(\lam)$ to be the set of zeros of $G(\lam).$ In other words, we set \be\label{spec} \poles(G) := \{ \lam \in \C : \psi_G(\lam) =0 \} \,\, \mbox{ and } \,\, \sp(G) :=\{ \lam \in \C : \phi_G(\lam) = 0\}.\ee
The spectrum of the transfer function of an LTI system is called  the {\em transmission zeros} of the LTI system~\cite{rosenbrock70}. Thus, the transmission zeros of the LTI system $\Sigma$ in (\ref{rassf}) are given by the spectrum of the transfer function $G(\lam) := P(\lam) + C(\lam E - A)^{-1}B.$

\begin{definition}[Eigenvalue and eigenpole] Let $ G(\lam) \in \C(\lam)^{m\times n}.$ A complex number $\lam_0 \in \C$ is said to be an {\bf eigenvalue} of $G(\lam)$ if $\rank(G(\lam_0)) < \nrank(G).$   We denote the {\em eigenspectrum} (set of eigenvalues) of $G(\lam)$ by $ \mathrm{Eig}(G).$

A complex number $\lam_0$ is said to be an {\bf eigenpole} of $G(\lam)$ if $\lam_0$ is a pole of $G(\lam)$ and there exists $v(\lam) \in \C^{n}[\lam]$ such that $v(\lam_0) \neq 0$ and  $\lim_{\lam \rightarrow \lam_{0}} G(\lam)v(\lam) = 0. $  We denote the set of eigenpoles of $G(\lam)$ by $ \mathrm{Eip}(G).$
\end{definition}

We mention that $G(\lam)$ is said to have an eigenvalue (eigenpole) at $\lam = \infty$ if $ G(1/\lam)$ has an eigenvalue (eigenpole) at $ \lam = 0.$  For example,  $G(\lam)$ in (\ref{exnoevl}) has no eigenvalue but has an eigenpole at $ \lam =2.$ On the other hand,  $G(\lam) := \diag(\lam, 1/\lam)$ has no eigenvalue but has eigenpoles at $0$ and $\infty.$ We, however, consider only finite eigenvalues and eigenpoles of rational matrix functions.

\begin{theorem} \label{eigpole}
Let $G(\lam) \in \C(\lam)^{m \times n}. $ Then $\sp(G) = \mathrm{Eig}(G) \cup \mathrm{Eip}(G)$ and
 \beano \mathrm{Eig}(G) &=& \{ \lam \in \C : \phi_G(\lam) = 0 \mbox{ and } \psi_G(\lam) \neq 0\}, \\
\mathrm{Eip}(G) &=& \{ \lam \in \C : \phi_G(\lam) = 0 \mbox{ and } \psi_G(\lam) = 0\},\eeano
where $\phi_G(\lam)$ and $\psi_G(\lam)$  are zero and pole polynomials of $G(\lam),$ respectively.
\end{theorem}

\begin{proof}
Suppose that $\mathbf{SM}(G(\lam)) = \diag \left(\phi_{1}(\lam)/\psi_{1}(\lam), \ldots, \phi_{k}(\lam)/\psi_{k}(\lam), 0, \ldots, 0\right).$ Then  $\phi_{G}(\lam) = \prod^k_{j=1}\phi_j(\lam)$ and $\psi_{G}(\lam) = \prod^k_{j=1}\psi_j(\lam)$. Suppose that $\lam_0 \in \mathrm{Eig}(G).$ Then we have $ \rank (G(\lam_0)) < k = \nrank(G).$ Consequently, $ \phi_{i}(\lam_0) = 0$ for some $i$ and hence  $ \phi_{G}(\lam_0) = 0.$  Since  $ \phi_i$ and $\psi_i$ are coprime, it follows that $ \psi_i(\lam_0) \neq 0$ for $i= 1:k$ and hence $ \psi_G(\lam_0) \neq 0.$ Conversely, if $\phi_G(\lam_0) =0$ and $ \psi_G(\lam_0) \neq 0$ then $ \phi_i(\lam_0) $ for some $i$ and hence $\rank(G(\lam_0)) < k = \nrank(G)$ showing that $ \lam_0 \in \mathrm{Eig}(G).$

Next, let $\lam_0 \in \mathrm{Eip}(G).$ Then there exists $v(\lam) \in \C^{n}[\lam]$ such that $v(\lam_0)\neq 0$ and $\lim_{\lam \rightarrow \lam_{0}} G(\lam)v(\lam) = 0. $  By Smith-McMillan form,  $G(\lam) = U(\lam)\,\mathbf{SM}(G(\lam))\, V(\lam) $ for some unimodular matrix polynomials $U(\lam)$ and $V(\lam).$ Hence
\begin{equation}
U^{-1}(\lam) G(\lam) v(\lam) = \mathbf{SM}(G(\lam)) V(\lam) v(\lam) . \label{spgzoq}
\end{equation}
Setting  $ x(\lam) := V(\lam)v(\lam)$ it follows that $x(\lam) \rightarrow x(\lam_{0})= V(\lam_{0})v(\lam_{0}) \neq 0$. Hence by  (\ref{spgzoq}) we have
$\frac{\phi_{i}(\lam)}{\psi_{i}(\lam)} x_{i}(\lam) \rightarrow 0$ for $i = 1:k$. Since $x(\lam_{0}) \neq 0$, we have  $x_{i}(\lam_{0}) \neq 0$ for  some $i$. This shows that as $\lam \rar \lam_0,$ we have
$$ \frac{\phi_{i}(\lam)}{\psi_{i}(\lam)}= \frac{1}{x_{i}(\lam)} \frac{\phi_{i}(\lam)}{\psi_{i}(\lam)} x_{i}(\lam) \rar 0. $$
Since $\phi_{i}$ and $\psi_{i}$ are coprime, $\psi_{i}(\lam_{0}) \neq 0$ and $\phi_{i}(\lam) \rightarrow \phi_{i}(\lam_{0}) = 0$ for some $i$. Hence $\phi_{G}(\lam_0) = 0.$ Also $ \psi_G(\lam_0) =0$ since $\lam_0$ is a pole.

Conversely, suppose that $\phi_{G}(\lam_0) = 0$ and $\psi_G(\lam_0) =0.$ Define $u(\lam) := V^{-1}(\lam) e_{i}, $ where $U(\lam)$ and $V(\lam)$ are as in (\ref{spgzoq}). Then $$G(\lam) u(\lam) = U(\lam) \mathbf{SM}(G(\lam)) e_{i} = U(\lam) \frac{\phi_{i}(\lam)}{\psi_{i}(\lam)} e_{i} \rightarrow 0 \text{ as }\lam \rightarrow \lam_{0},$$ which shows that $\lam_{0} \in \mathrm{Eip}(G).$  Hence the proof.
\end{proof}

Thus the spectrum of a rational matrix function $G(\lam)$ consists of eigenvalues and eigenpoles. We define partial multiplicities of zeros (eigenvalues and eigenpoles) of a rational matrix function as follows.

\begin{definition}[Partial multiplicities of zeros] Let $ G(\lam) \in \C(\lam)^{m\times n}$ and $ \phi_1(\lam), \ldots, \phi_k(\lam)$ be invariant zero polynomials of $G(\lam).$ Let $\lam_0 \in \sp(G). $ Then $\phi_{G}(\lam_0) = 0$ and $\phi_{i}(\lam) = (\lam - \lam_0)^{\gamma_i}d_{i}(\lam)$ with $d_{i}(\lam_0) \neq 0$ and $\gamma_i \geq 0$ for $i = 1:k. $ The index tuple $\ind_{\phi}(\lam_0, G) := (\gamma_1, \ldots, \gamma_k)$  is called the  multiplicity index of $G$ at $\lam_0$  and satisfies the condition $0\leq \gamma_1 \leq \gamma_2 \leq \cdots \leq \gamma_k.$ The nonzero components in $\ind_{\phi}(\lam_0, G)$ are called the partial multiplicities of $\lam_0$ as a zero of $G(\lam). $ The factors $(\lam - \lam_0)^{\gamma_i}$ with $\gamma_i \neq 0$ are called the  elementary divisors of $G(\lam)$ at $\lam_0$. The algebraic multiplicity of $\lam_0$ is defined by
\begin{align*}
a_{\phi}(\lam_0) &:= \gamma_1 + \gamma_2+ \cdots+ \gamma_k
= \text{multiplicity of } \lam_0 \text{ as a root of } \phi_{G}(\lam).
\end{align*}
If $a_{\phi}(\lam_0) = 1$ then $\lam_0$ is called a simple zero of $G(\lam). $\\
\end{definition}

Similarly, we define  partial multiplicities of poles of a rational matrix function as follows.

\begin{definition}[Partial multiplicities of poles] Let $ G(\lam) \in \C(\lam)^{m\times n}$ and $ \psi_1(\lam), \ldots, \psi_k(\lam)$ be invariant pole polynomials of $G(\lam).$ Let $\lam_0 \in \poles(G)$. Then $\psi_{G}(\lam_0) = 0 $ and $\psi_{i}(\lam) = (\lam - \lam_0)^{\alpha_i}q_{i}(\lam)$ with $q_{i}(\lam_0) \neq 0$ and $\alpha_i \geq 0$ for $i = 1, 2, \ldots, k. $ The index tuple $\ind_{\psi}(\lam_0, G) := (\alpha_k, \alpha_{k-1} \ldots, \alpha_1)$ is called the  multiplicity index  of $G$ at $\lam_0$ and satisfies the condition $\alpha_k \leq  \alpha_{k-1} \leq \cdots \leq  \alpha_1$. The nonzero components in $\ind_{\psi}(\lam_0, G)$ are called the partial multiplicities of $\lam_0$ as a pole of $G(\lam). $ The factors $(\lam - \lam_0)^{\alpha_i}$ with $\alpha_i \neq 0$ are called the  elementary divisors of $G(\lam)$ at the pole $\lam_0$. The  algebraic multiplicity of $\lam_0$ is defined by
\begin{align*}
a_{\psi}(\lam_0) &:= \alpha_1 + \alpha_2+ \cdots+ \alpha_k
 = \text{multiplicity of } \lam_0 \text{ as a root of } \psi_{G}(\lam).
\end{align*}
If $a_{\psi}(\lam_0) = 1$ then $\lam_0$ is called a  simple pole of $G(\lam). $ \\
\end{definition}

 We remark that if $ \lam_0 \in \mathrm{Eip}(G)$ then  $\ind_{\phi}(\lam_0, G) \neq 0$ and $\ind_{\psi}(\lam_0, G) \neq 0$. In such a case, $\lam_0$ has two multiplicity indices, namely, the multiplicity index as a zero and the multiplicity index as a pole. In general $\ind_{\phi}(\lam_0, G) \neq \ind_{\psi}(\lam_0, G). $

\begin{exam}
Consider $G(\lam) := \left[
                      \begin{array}{cc}
                        \frac{1}{\lam (\lam-2)^2} &  \\
                         & \frac{\lam -2}{\lam} \\
                      \end{array}
                    \right]. $ Then $ \lam = 2$ is an eigenpole with  $\ind_{\phi}(2, G) = (0, 1)$ and $\ind_{\psi}(2, G) = (0, 2). $ Hence $\ind_{\phi}(2, G) \neq  \ind_{\psi}(2, G).$
 $\blacksquare$ \\
\end{exam}

\section{Fiedler pencils for Rosenbrock system polynomial} For the rest of the paper, we consider the linear time-invariant (LTI) system $\Sigma$ given in (\ref{rassf}), that is,
$$
\begin{array}{ll} \Sigma:  &
\begin{array}{ll}
 E \dot{x}(t) = A x(t) + Bu(t) \\
 y(t) = C x(t) + P(\frac{d}{dt}) u(t),
\end{array}
\end{array}
$$
where $E \in \C^{r\times r}$ is nonsingular and $P(\lam) \in \C^{n \times n}[\lam]$. Also, unless stated otherwise,
\be \label{str}
\mathcal{S}(\lam) := \left[
                       \begin{array}{c|c}
                        P(\lam) & C \\
                        \hline
                         B & (A - \lam E) \\
                          \end{array}
                          \right] \mbox{ and }  G(\lam) := P(\lam) + C(\lam E - A)^{-1}B\ee
will denote the Rosenbrock system polynomial and the transfer function
of the LTI system $\Sigma,$ respectively. We also refer to $\mathcal{S}(\lam)$ as the Rosenbrock system matrix of the LTI system $\Sigma.$
We define the {\bf degree} of the Rosenbrock system polynomial $\mathcal{S}(\lam)$ to be the degree of the matrix polynomial $P(\lam).$ For the rest of the paper, we assume that the degree of $P(\lam)$ is $m$ and $P(\lam)$  is given by
\be \label{poly} P(\lam) := \sum_{j=0}^{m}\lam^{j} A_j \,\,\,\,  \mbox{ with } A_m \neq 0. \ee

We now introduce Fiedler pencils of $\mathcal{S}(\lam)$  which are potential candidates for {\em ``trimmed structured linearizations"} of $\mathcal{S}(\lam)$ and  describe their constructions. We mention that although the pencil $A-\lam E$ in $\mathcal{S}(\lam)$ is assumed to be regular and $E$ to be nonsingular,  which are important for the associated LTI state-space system, these assumptions are irrelevant for construction of Fiedler pencils of $\mathcal{S}(\lam)$. The construction of a Fiedler pencil described below holds good even when $\mathcal{S}(\lam)$ and  $P(\lam)$ are singular polynomials and $A-\lam E$ is  a singular pencil.

For a ready reference, we briefly describe Fiedler matrices associated with $P(\lam)$  and refer to~\cite{TDM} for further details. We closely follow the notational conventions used in~\cite{TDM}. The  $nm \times  nm$  matrices $ M_0,\ldots, M_m$  given by
\begin{equation}
 M_{0} := \left[
                   \begin{array}{cc}
                     I_{(m-1)n} &  \\
                      & -A_{0} \\
                   \end{array}
                 \right] \label{0mmfp}, \,\, M_{m} := \left[
          \begin{array}{cc}
            A_{m} &  \\
             & I_{(m-1)n} \\
          \end{array}
        \right]
\end{equation}
and
\begin{equation}
M_{i} := \left[
  \begin{array}{cccc}
    I_{(m-i-1)n} &  & & \\
     & -A_{i} & I_{n} & \\
     & I_{n} & 0  &  \\
     &   &   & I_{(i-1)n}\\
  \end{array}
\right],  \indent i= 1, \ldots, m-1, \label{imfp}
\end{equation}
are called Fiedler matrices associated with $P(\lam).$  The Fiedler matrices $M_i$ and $M_j$ commute when $|i-j| >1,$ that is,
  $M_{i}M_{j} = M_{j}M_{i}$ when $|i - j| > 1.$ Further, for $i= 1, \ldots, m-1$,  each $M_{i}$ is invertible and
\begin{equation}
M_{i}^{-1} = \left[
  \begin{array}{cccc}
    I_{(m-i-1)n} &  & & \\
     & 0 & I_{n} & \\
     & I_{n} & A_{i}  &  \\
     &   &   & I_{(i-1)n} \\
  \end{array}
\right]. \label{ifp}
\end{equation}

Let  $\sigma : \{0, 1, \ldots, m-1\}\rightarrow \{ 1, 2, \ldots, m\}$ be a bijection. Then $M_{\sigma}$ denotes the product of Fiedler matrices given by
$$M_{\sigma} := M_{\sigma^{-1}(1)}M_{\sigma^{-1}(2)}\cdots M_{\sigma^{-1}(m)}. $$  Note that $\sigma(i)$ describes the position of the factor $M_{i}$ in the product $M_{\sigma}$, that is, $\sigma(i) = j$ means that $M_{i}$ is the $j$-th factor in the product. For convenience, we set   $M_{\emptyset} := I_{nm},$ where $\emptyset$ is the empty set. The $nm\times nm$ pencil $L_\sigma(\lam)$ given by  $$L_{\sig}(\lam) := \lam M_m - M_\sig$$ is called the Fiedler pencil of $P(\lam)$ associated with $\sig,$ see~\cite{TDM, AV}.

We now introduce Fiedler-like matrices associated with the Rosenbrock system polynomial $\mathcal{S}(\lam)$ and construct Fiedler-like pencils for $\mathcal{S}(\lam).$  Recall from (\ref{str}) and (\ref{poly}) that $\mathcal{S}(\lam)$ is an $(n+r)\times (n+r)$ matrix polynomial of degree $m.$

\begin{definition}[Fiedler matrices]  Define the  $(nm +r) \times (nm+ r)$ matrices $\mathbb{M}_0, \ldots, \mathbb{M}_m$ by
\begin{equation}
\mathbb{M}_{0} := \left[
                       \begin{array}{c|c}
                         M_{0} & -e_{m} \otimes C \\
                         \hline
                         -e_{m}^{T}\otimes B & -A \\
                       \end{array}
                     \right], \indent \mathbb{M}_{m} := \left[
                                        \begin{array}{c|c}
                                          M_{m} & 0 \\
                                          \hline
                                          0 & -E \\
                                        \end{array}
                                      \right] \label{0mmfr}
\end{equation}
and
\begin{equation}
\mathbb{M}_{i} := \left[
                       \begin{array}{c|c}
                         M_{i} & 0 \\
                         \hline
                         0 & I_{r} \\
                       \end{array}
                     \right], \indent i = 1, \ldots, m-1, \label{imfr}
\end{equation}
where $M_{i}, i= 0, 1, \ldots, m$ are Fiedler matrices associated with $P(\lambda)$. We refer to the matrices $\mathbb{M}_{i}, i= 0, 1, \ldots, m,$ as the  Fiedler matrices  associated with $\mathcal{S}(\lam)$.
\end{definition}

Observe that unlike in the case of Fiedler matrices associated with  $P(\lam)$, the Fiedler matrices $\mathbb{M}_0$ and $\mathbb{M}_m$ do not commute, that is,
$\mathbb{M}_0 \mathbb{M}_m  \neq \mathbb{M}_m\mathbb{M}_0.$  However, $\mathbb{M}_i$ commutes with $\mathbb{M}_j$ when $|i-j| >1$  and $|i-j| \neq m,$ that is,
\begin{equation}
\mathbb{M}_{i} \mathbb{M}_{j} = \mathbb{M}_{j} \mathbb{M}_{i} \indent\text{  for   } |i-j| > 1, \text{ except for } \mathbb{M}_{m} \text{   and  } \mathbb{M}_{0}. \label{crr}
\end{equation}
For $i= 1:m-1,$ each  $\mathbb{M}_{i}$ is invertible and
\begin{equation}
\mathbb{M}_{i}^{-1} = \left[
                             \begin{array}{c|c}
                               M_{i}^{-1} & 0 \\
                               \hline
                               0 & I_{r} \\
                             \end{array}
                           \right], \label{ifr} \mbox{   where  } M_{i}^{-1} \mbox{ is given in } (\ref{ifp}).
\end{equation}
Hence we have
\begin{equation}
\mathbb{M}_{i}^{-1} \mathbb{M}_{j}^{-1} = \mathbb{ M}_{j}^{-1} \mathbb{M}_{i}^{-1}  \indent \mbox{   for  } |i-j| > 1, \text{ except for } \mathbb{M}_{0}^{-1} \text{  and  } \mathbb{M}_{m}^{-1}. \label{icrr}
\end{equation}

\begin{definition}[Fiedler pencil] Let $\sigma : \{0, 1, \ldots, m-1\}\rightarrow \{ 1, 2, \ldots, m\}$ be a bijection.  Then the $(mn+r)\times (mn+r)$ matrix pencil $\mathbb{L}_{\sigma}(\lambda)$ given by
\begin{equation}
\mathbb{L}_{\sigma}(\lambda) := \lambda \mathbb{M}_{m} - \mathbb{M}_{\sigma^{-1}(1)}\mathbb{M}_{\sigma^{-1}(2)}\cdots \mathbb{M}_{\sigma^{-1}(m)} = \lambda \mathbb{M}_{m} - \mathbb{M}_{\sigma} \label{fpr}
\end{equation}
is said to be the Fiedler pencil of the Rosenbrock system polynomial $\mathcal{S}(\lam)$ associated with $\sig$.
\end{definition}

If $\sigma$ and $\tau$ are two bijections from $\{0, 1, \ldots, m-1\}$  to $ \{ 1, 2, \ldots, m\}$ then because of commutation relation we may have $\mathbb{M}_\sigma = \mathbb{M}_\tau.$ Consequently, a Fiedler pencil may be associated with more than one  bijection.

\begin{definition}
Let $\sig, \tau :  \{0, 1, \ldots, m-1\}\rightarrow \{ 1, 2, \ldots, m\}$ be bijections. Then $\sig$ is said to be equivalent to $\tau$ (denoted as $\sig\sim \tau$) if $\mathbb{M}_{\sigma} = \mathbb{M}_{\tau}. $
\end{definition}

\begin{exam}
Consider  $G(\lam) := \lam^{4} A_{4} + \cdots + A_{0} + C(\lam E - A)^{-1} B$ and the Fiedler pencil $\mathbb{L}_{\sig}(\lam) := \lam \mathbb{M}_{4} - \mathbb{M}_{1}\mathbb{M}_{0}\mathbb{M}_{2}\mathbb{M}_{3}$ with $\sigma^{-1} =( 1, 0, 2, 3).$  Then
$$\mathbb{L}_{\sig}(\lam) = \lam \left[
                                    \begin{array}{cc|c}
                                      A_{4} &   &   \\
                                        & I_{3n} &   \\
                                      \hline
                                        &   &  -E \\
                                    \end{array}
                                  \right] -
                               \left[
                               \begin{array}{cccc|c}
                                 -A_{3} & I_{n} & 0 & 0 & 0  \\
                                 -A_{2} & 0 & I_n & 0 & 0 \\
                                 -A_{1} & 0 & 0 & -A_{0} & -C \\
                                 I_n & 0 & 0 & 0 & 0 \\
                                 \hline
                                  0 &  0 & 0  & -B & -A \\
                               \end{array}
                             \right].
$$

Considering $\mathbb{L}_{\tau}(\lambda)  := \lambda \mathbb{M}_{4} - \mathbb{M}_{2}\mathbb{M}_{0}\mathbb{M}_{1}\mathbb{M}_{3} $ with $\tau^{-1} = (2, 0, 1, 3),$  we have
$$\mathbb{L}_{\tau}(\lambda) = \left[
      \begin{array}{cccc|c}
        A_{4} &  &  &  &  \\
         & I_{n} &  &  &  \\
         &  & I_{n} &  &  \\
         &  &  & I_{n} &  \\
        \hline
         &  &  &  & -E \\
      \end{array}
    \right] - \left[
                \begin{array}{cccc|c}
                  -A_{3} & I_{n} & 0 & 0 & 0 \\
                  -A_{2} & 0 & -A_{1} & I_{n} & 0 \\
                  I_{n} & 0 & 0 & 0 & 0 \\
                  0 & 0 & -A_{0} & 0 & -C \\
                  \hline
                  0 & 0 & -B & 0 & -A \\
                \end{array}
              \right].
$$
For the bijection $\delta$ given by $\delta^{-1} = (0,2, 3, 1), $ we have $ \mathbb{L}_{\tau}(\lam) = \mathbb{L}_{\delta}(\lam).$
This shows that  $\delta\sim \tau$.        $\blacksquare$
\end{exam}

\begin{definition}\label{dci} \cite{TDM}
Let $\sigma : \{0, 1, \ldots, m-1\} \rightarrow \{1, 2, \ldots, m\}$ be a bijection.
\begin{itemize}

\item[(1)] For $d = 0, \ldots, m-2$, we say that $\sigma$ has a consecution at $d$ if $\sigma(d) < \sigma(d+1)$ and $\sigma$ has an inversion at $d$ if $\sigma(d) > \sigma(d+1)$.

\item[(2)] The tuple CISS$(\sigma) :=(c_{1}, i_{1}, c_{2}, i_{2}, \ldots, c_{l}, i_{l})$ is called the consecution-inversion structure sequence of $\sigma$, where $\sigma$ has $c_{1}$ consecutive consecutions at $0, 1, \ldots, c_{1}-1;$ $i_{1}$ consecutive inversions at $c_{1}, c_{1}+1, \ldots, c_{1}+i_{1}-1$ and so on, up to $i_{l}$ inversions at $m-1-i_{l}, \ldots, m-2$.

\item[(3)] We denote the total number of consecutions and inversions in $\sigma$ by $c(\sigma)$ and $i(\sigma),$ respectively. Then $c(\sig) = \sum\limits_{j=1}^{l}c_{j}$, \, $i(\sig) = \sum\limits_{j=1}^{l}i_j$, and $c(\sig) + i(\sig) = m-1. $
\end{itemize}
\end{definition}

Said differently,  if $\sig : \{0, 1, \ldots, m-1\} \rightarrow \{1, 2, \ldots, m\}$ is a bijection then
$\sigma$ has a consecution at $d$ if and only if $\mathbb{M}_{d}$ is to the left of $\mathbb{M}_{d+1}$ in $\mathbb{M}_{\sigma}$, while $\sigma$ has an inversion at $d$ if and only if $\mathbb{M}_{d}$ is to the right of $\mathbb{M}_{d+1}$ in $\mathbb{M}_{\sigma}$.
Further, we have $$M_i M_j = M_j M_i \Longleftrightarrow \mathbb{M}_i \mathbb{M}_j = \mathbb{M}_j \mathbb{M}_i \,\, \mbox{ for } \,\, i, j \in \sig.$$

We now show that a  Fiedler pencil of $\mathcal{S}(\lam)$  can be constructed from a Fiedler pencil $P(\lam)$ and vice-versa. In fact, there is a bijection from the set of Fiedler pencils of $\mathcal{S}(\lam)$ to the set of Fiedler pencils of $P(\lam).$

\begin{theorem}\label{fflptflr}
Let $\sigma :\{0, 1, \ldots, m-1\} \rightarrow \{1, 2, \ldots, m\}$ be a bijection. Let $L_{\sig}(\lam)$ and $\mathbb{L}_{\sig}(\lam)$ be the Fiedler pencils of $P(\lam)$ and $\mathcal{S}(\lam)$, respectively, associated with $\sigma,$ that is, $L_{\sigma}(\lam) := \lam M_{m} - M_{\sigma}$ and $\mathbb{L}_{\sigma}(\lam) := \lam \mathbb{M}_{m} - \mathbb{M}_{\sigma}$. If $\sig^{-1} = (\sig_1^{-1}, 0, \sig_2^{-1})$ for some bijections $\sigma_{1}$ and $  \sigma_{2},$ then
$$ \mathbb{L}_{\sigma}(\lam) = \left[
      \begin{array}{c|c}
        L_{\sigma}(\lambda) & M_{\sigma_{1}}(e_{m}\otimes C) \\
        \hline
        (e_{m}^{T}\otimes B) M_{\sigma_{2}} & (A-\lam E) \\
      \end{array}
    \right]. $$
Further, if CISS$(\sig) = (c_1, i_1, \ldots, c_l, i_l)$ then
$$\mathbb{L}_{\sig}(\lam) = \left[
                               \begin{array}{c|c}
                               L_{\sig}(\lam) & e_m \otimes C \\
                               \hline
                               e_{(m-c_1)}^{T} \otimes B & A-\lam E \\
                               \end{array}
                               \right], \indent \text{ if } c_1 >0
$$ and $$\mathbb{L}_{\sig}(\lam) = \left[
                               \begin{array}{c|c}
                               L_{\sig}(\lam) & e_{(m-i_1)} \otimes C \\
                               \hline
                               e_{m}^{T} \otimes B & A-\lam E \\
                               \end{array}
                               \right],  \indent \text{ if } c_1 = 0. $$
Thus the map $\mathrm{Fiedler}(P) \longrightarrow \mathrm{Fiedler}(\mathcal{S}), $ $ \lam M_m -M_{\sig} \longmapsto \lam \mathbb{M}_m -\mathbb{M}_{\sig}$ is a bijection, where $\mathrm{Fiedler}(P)$ and $\mathrm{Fiedler}(\mathcal{S})$, respectively, denote the set of Fiedler pencils of $P(\lam)$ and $\mathcal{S}(\lam)$.
\end{theorem}

\begin{proof}
 We have $\mathbb{L}_{\sig}(\lam) = \lam \mathbb{M}_{m} - \mathbb{M}_{\sig} = \lam \mathbb{M}_{m} - \mathbb{M}_{\sig_{1}} \mathbb{M}_{0}\mathbb{M}_{\sig_{2}} $
\beano  &=& \lam \left[
              \begin{array}{c|c}
               M_{m} & 0 \\
                \hline
                0 & -E \\
              \end{array}
            \right] - \left[
                        \begin{array}{c|c}
                          M_{\sigma_{1}} & 0  \\
                          \hline
                           0 & I_{r}\\
                        \end{array}
                      \right]\left[
                               \begin{array}{c|c}
                                 M_{0} & -e_{m} \otimes C \\
                                 \hline
                                -e_{m}^{T} \otimes B & -A \\
                               \end{array}
                             \right]\left[
                                      \begin{array}{c|c}
                                        M_{\sigma_{2}} & 0 \\
                                        \hline
                                        0 & I_{r} \\
                                      \end{array}
                                    \right] \\
& = & \lam \left[
              \begin{array}{c|c}
               M_{m} & 0 \\
                \hline
                0 & -E \\
              \end{array}
            \right] - \left[
                        \begin{array}{c|c}
                      M_{\sigma_{1}} M_{0} M_{\sigma_{2}} & - M_{\sigma_{1}}(e_{m} \otimes C) \\
                          \hline
                         (-e_{m}^{T} \otimes B) M_{\sigma_{2}}& -A \\
                        \end{array}
                      \right] \\
& =& \left[
      \begin{array}{c|c}
        L_{\sigma}(\lambda) & M_{\sigma_{1}}(e_{m}\otimes C) \\
        \hline
        (e_{m}^{T}\otimes B) M_{\sigma_{2}} & (A-\lambda E) \\
      \end{array}
    \right].
\eeano
Now suppose that CISS$(\sig) = (c_1, i_1, \ldots, c_l, i_l)$.

Case $I:$ Suppose that $c_1 > 0$. Then by commutativity relation we have $\mathbb{M}_{\sig} = \mathbb{M}_{\sig_1}\mathbb{M}_0 \mathbb{M}_1 \cdots \mathbb{M}_{c_1}$ with $c_1+1 \in \sig_1$. Thus $\mathbb{M}_{\sig} = \mathbb{M}_{\sig_1} \mathbb{M}_0 \mathbb{M}_{\sig_2}$, where $\mathbb{M}_{\sig_2} = \mathbb{M}_1 \cdots \mathbb{M}_{c_1}$. Hence
\beano \mathbb{M}_{\sig}  &=& \left[
       \begin{array}{c|c}
         M_{\sig_1} &  \\
         \hline
          & I_r \\
       \end{array}
     \right]\left[
              \begin{array}{c|c}
                M_0 & -e_m \otimes C \\
                \hline
                -e_m^{T} \otimes B & -A \\
              \end{array}
            \right]\left[
                     \begin{array}{c|c}
                       M_{\sig_2} &  \\
                       \hline
                        & I_r \\
                     \end{array}
                   \right] \\  &=& \left[
                               \begin{array}{c|c}
                                 M_{\sig_1} M_0 M_{\sig_2} & M_{\sig_1} (-e_m \otimes C) \\
                                 \hline
                                 (-e_m^{T} \otimes B)M_{\sig_2} & -A \\
                               \end{array}
                             \right].
\eeano
Since $j \in \sig_1$ implies that $j \geq c_1+1$, we have $\mathbb{M}_{\sig_1} = \left[
                                                                                  \begin{array}{c|c}
                                                                                  * &  \\
                                                                                  \hline
                                                                                  & I_{c_1n} \\
                                                                                  \end{array}
                                                                                  \right]
$. This shows that $\mathbb{M}_{\sig_1} (e_m \otimes I_n) = e_m \otimes I_n$ and  $ M_{\sig_1} (-e_m \otimes C) = -e_m \otimes C$. Next, we have
\beano (e_m^{T} \otimes I_n) \mathbb{M}_1 &=& (e_m^{T} \otimes I_n) \left[
                                                               \begin{array}{ccc}
                                                                 I_{(m-2)n} &  &  \\
                                                                  & -A_1 & I_n \\
                                                                  & I_n & 0 \\
                                                               \end{array}
                                                             \right] = (e_{m-1}^{T} \otimes I_n),\\
(e_m^{T} \otimes I_n) \mathbb{M}_1 \mathbb{M}_2 &=& (e_{m-1}^{T} \otimes I_n) \left[
                                                               \begin{array}{cccc}
                                                                 I_{(m-3)n} &  & & \\
                                                                  & -A_2 & I_n & \\
                                                                  & I_n & 0 & \\
                                                                   &  &  & I_n \\
                                                               \end{array}
                                                             \right] = (e_{m-2}^{T} \otimes I_n) \eeano and so on. Thus
$(e_m^{T} \otimes I_n) \mathbb{M}_1 \mathbb{M}_2 \cdots \mathbb{M}_{c_1} = (e_{m-c_1}^{T} \otimes I_n)$. Hence $(e_m^{T} \otimes I_n) \mathbb{M}_{\sig_2} = (e_{m-c_1}^{T} \otimes I_n)$ and $(-e_m^{T} \otimes B)M_{\sig_2} = -(e_{m-c_1}^{T} \otimes B)$. Consequently, we have
$$\mathbb{L}_{\sig}(\lam) = \lam \mathbb{M}_{m} - \mathbb{M}_{\sig} = \left[
                               \begin{array}{c|c}
                               -L_{\sig}(\lam) & e_m \otimes C \\
                               \hline
                               e_{(m-c_1)}^{T} \otimes B & A-\lam E \\
                               \end{array}
                               \right].$$

Case $II:$ Suppose that $c_1 = 0$. Then $\sig$ has $i_1$ inversions at $0$. Hence by commutativity relations we have $\mathbb{M}_{\sig} = \mathbb{M}_{i_1} \cdots \mathbb{M}_1 \mathbb{M}_0 \mathbb{M}_{\sig_2} =:\mathbb{M}_{\sig_1}\mathbb{M}_0 M_{\sig_2} $ with $i_1+1 \in \sig_2$.  Hence
\beano \mathbb{M}_{\sig}  &=& \left[
       \begin{array}{c|c}
         M_{\sig_1} &  \\
         \hline
          & I_r \\
       \end{array}
     \right]\left[
              \begin{array}{c|c}
                M_0 & -e_m \otimes C \\
                \hline
                -e_m^{T} \otimes B & -A \\
              \end{array}
            \right]\left[
                     \begin{array}{c|c}
                       M_{\sig_2} &  \\
                       \hline
                        & I_r \\
                     \end{array}
                   \right]\\ & = & \left[
                               \begin{array}{c|c}
                                 M_{\sig_1} M_0 M_{\sig_2} & M_{\sig_1} (-e_m \otimes C) \\
                                 \hline
                                 (-e_m^{T} \otimes B)M_{\sig_2} & -A \\
                               \end{array}
                             \right].\eeano

Since $j \in \sig_2$ implies that $j \geq i_1+1$, we have $\mathbb{M}_{\sig_2} = \left[
                         \begin{array}{c|c}
                         * &  \\
                         \hline
                         & I_{i_1n} \\
                         \end{array}
                         \right]
$. This shows that $ (e_m^{T} \otimes I_n)\mathbb{M}_{\sig_2} = e_m^{T} \otimes I_n$. Hence $ M_{\sig_2} (-e_m^{T} \otimes B) = -e_m^{T} \otimes B$. Next, we have
\beano  \mathbb{M}_1 (e_m \otimes I_n) &=& \left[
                                     \begin{array}{ccc}
                                     I_{(m-2)n} &  &  \\
                                     & -A_1 & I_n \\
                                     & I_n & 0 \\
                                     \end{array}
                                     \right]  (e_m \otimes I_n)= (e_{m-1} \otimes I_n),\\
\mathbb{M}_2 \mathbb{M}_1 (e_m \otimes I_n) &=& \left[
                                                               \begin{array}{cccc}
                                                                 I_{(m-3)n} &  & & \\
                                                                  & -A_2 & I_n & \\
                                                                  & I_n & 0 & \\
                                                                   &  &  & I_n \\
                                                               \end{array}
                                                             \right] (e_{m-1} \otimes I_n) = (e_{m-2} \otimes I_n). \eeano  Thus
$ \mathbb{M}_{i_1}  \cdots \mathbb{M}_2\mathbb{M}_1 (e_m \otimes I_n)= (e_{m-i_1} \otimes I_n)$. Hence $\mathbb{M}_{\sig_1}(e_m \otimes I_n) = (e_{(m-i_1)} \otimes I_n)$ and $M_{\sig_1}(-e_m \otimes C) = -(e_{(m-i_1)} \otimes C)$. Consequently, we have
$$\mathbb{L}_{\sig}(\lam) = \lam \mathbb{M}_{m} - \mathbb{M}_{\sig} = \left[
                               \begin{array}{c|c}
                               -L_{\sig}(\lam) & e_{m-i_1} \otimes C \\
                               \hline
                               e_{m}^{T} \otimes B & A-\lam E \\
                               \end{array}
                               \right]. $$
Recall that for each $i, j \in \sig$, we have $M_i M_j = M_j M_i \Leftrightarrow \mathbb{M}_i \mathbb{M}_j = \mathbb{M}_j \mathbb{M}_i$. Hence it follows that $\# (\mathrm{Fiedler}(P)) = \#(\mathrm{Fiedler}(\mathcal{S}))$. This completes the proof.
\end{proof}

The companion pencil $\mathcal{C}(\lam)$ given in (\ref{compr}) is in fact a Fiedler pencil of $\mathcal{S}(\lam).$

\begin{proposition} \label{compan} Let $\mathcal{C}(\lam)$ be the companion pencil given in (\ref{compr}). Then $$ \mathcal{C}(\lambda)  =  \lambda \mathbb{M}_{m} - \mathbb{M}_{m-1}\mathbb{M}_{m-2}\cdots \mathbb{M}_{1}\mathbb{M}_{0} = \mathbb{L}_\sigma(\lam),$$ where $\sig^{-1} = (m-1, \ldots, 2, 1, 0).$

\end{proposition}

\begin{proof} Consider the Fiedler pencil $\mathbb{L}_\sigma(\lam)$ with  $\sig^{-1} = (m-1, \ldots, 2, 1, 0).$  Then it follows that CISS$(\sig) = (0, m-1).$ Hence by Theorem~\ref{fflptflr}, we have
$$ \mathbb{L}_{\sigma}(\lambda) =\left[
                             \begin{array}{c|c}
                               \lam M_m - M_{m-1}\cdots M_1M_0 & e_1\otimes C \\
                               \hline
                               e_m^{T} \otimes B & (A - \lam E) \\
                             \end{array}
                           \right] =
 \left[
                             \begin{array}{c|c}
                               C_P(\lam) & e_1\otimes C \\
                               \hline
                               e_m^{T} \otimes B & (A - \lam E) \\
                             \end{array}
                           \right],$$
 where $C_P(\lam)$ is the companion pencil of $P(\lam)$ given in (\ref{cpmpp}). The last equality follows from the fact that $C_P(\lam) = \lam M_m - M_{m-1}\cdots M_1M_0,$ see~\cite{TDM}.  This shows that $   \mathbb{L}_{\sigma}(\lambda) = \mathcal{C}(\lam)$  as desired.
\end{proof}

 The pencil $\mathcal{C}(\lam)$ is derived in \cite{bai11} and is referred to as companion linearization of $G(\lam)$. We refer to $\mathcal{C}(\lam)$ as the  {\rm first companion form} or the {\em first companion pencil}  of $\mathcal{S}(\lam).$ We also refer to $\mathcal{C}(\lam)$ as the  first companion pencil of the transfer function $G(\lam)$ associated with $\mathcal{S}(\lam)$.

We define the second companion form of $\mathcal{S}(\lam),$ which we denote by $ \mathcal{C}_2(\lam),$ by
$$\mathcal{C}_{2}(\lambda) =  \lambda \mathbb{M}_{m} - \mathbb{M}_{0}\mathbb{M}_{1}\cdots \mathbb{M}_{m-2}\mathbb{M}_{m-1} = \mathbb{L}_{\sigma}(\lambda),$$
where $\sigma^{-1} = (0, 1,  \ldots, m-1).$ Since  CISS$(\sig) = (m-1, 0)$ and  $ C_2(\lam) = \lam M_m - M_{0} \cdots M_{m-2} M_{m-1}$ is the second companion form of the polynomial $P(\lam)$, see~\cite{TDM},  by Theorem~\ref{fflptflr}, we have
\begin{align}
\mathcal{C}_{2}(\lam) &  =  \left[
                          \begin{array}{c|c}
                            C_{2}(\lam) & e_m \otimes C \\
                            \hline
                            e_1^{T} \otimes B & (A- \lam E) \\
                          \end{array}
                        \right] \nonumber \\
                       & = \lam \left[
                         \begin{array}{cccc|c}
                           A_{m} &  &  &  &  \\
                            & I_{n} &  &  &  \\
                            &  & \ddots &  &  \\
                            &  &  & I_{n} &  \\
                            \hline
                            &  &  &  & -E \\
                         \end{array}
                       \right]- \left[
                \begin{array}{cccc|c}
                  -A_{m-1} & I_{n} &  &  &   \\
                  -A_{m-2} & 0 & \ddots &  &  \\
                  \vdots & \ddots &  & I_{n} & \\
                  -A_{0} & \cdots & 0 & 0  & -C \\
                  \hline
                  -B  &    &    &   &  -A \\
                \end{array}
              \right].  \label{scfor}
\end{align}
We refer to $\mathcal{C}_{2}(\lam)$ as the {\em second companion form} of $\mathcal{S}(\lam)$ or the transfer function $G(\lam)$ associated with $\mathcal{S}(\lam)$.

It is well known that a matrix polynomial admits block pentadiagonal Fiedler pencils~\cite{TDM}. The next result characterizes block pentadiagonal Fiedler pencils of Rosenbrock system polynomials.

\begin{theorem}\label{pdpc} Let $\sigma :\{0, 1, \ldots, m-1\} \rightarrow \{1, 2, \ldots, m\}$ be a bijection. Let  $L_{\sigma}(\lambda) := \lambda M_{m} - M_{\sigma}$ and $\mathbb{L}_{\sigma}(\lam) := \lam \mathbb{M}_{m} - \mathbb{M}_{\sigma},$ respectively, be the Fiedler pencil of $P(\lam)$ and $\mathcal{S}(\lam)$ associated with $\sigma.$ Suppose that CISS$(\sigma) = (c_{1}, i_{1}, \ldots, c_{l}, i_{l}).$ Then $ \mathbb{L}_{\sigma}(\lam)$ is block pentadiagonal  if and only if $L_{\sigma}(\lam)$ is block pentadiagonal and $c_1 \leq 1$ and $i_1 \leq 1.$ In such a case, we have
$$ \mathbb{L}_{\sigma}(\lambda) = \left[
                             \begin{array}{c|c}
                                L_{\sigma}(\lambda) & e_m\otimes C \\
                              \hline
                            e_{m -1}^{T} \otimes B & (A - \lambda E) \\
                              \end{array}
                                  \right] \indent \mbox{ if } c_1>0 $$ $$\mbox{ and }\indent  \mathbb{L}_{\sigma}(\lambda) = \left[
                             \begin{array}{c|c}
                                L_{\sigma}(\lambda) & e_{m-1}\otimes C \\
                              \hline
                            e_{m}^{T} \otimes B & (A - \lambda E) \\
                              \end{array}
                                  \right] \indent \mbox{ if } c_1=0. $$
\end{theorem}

\begin{proof} The proof is immediate by Theorem \ref{fflptflr}. Indeed, for pentadiagonal  $\mathbb{L}_{\sigma}(\lam),$ we must have $c_1 \leq 1$ and $ i_1 \leq 1.$ If $c_1>0$ then $c_1 = 1$ and $ e_{m -c_{1}}^{T} \otimes B = e_{m -1}^{T} \otimes B$. So by the Theorem \ref{fflptflr} we have  $$\mathbb{L}_{\sigma}(\lambda) =\left[
                             \begin{array}{c|c}
                                L_{\sigma}(\lambda) & e_m\otimes C \\
                              \hline
                            e_{m -c_{1}}^{T} \otimes B & (A - \lambda E) \\
                              \end{array}
                                  \right]
                                    = \left[
                             \begin{array}{c|c}
                                L_{\sigma}(\lambda) & e_m\otimes C \\
                              \hline
                            e_{m -1}^{T} \otimes B & (A - \lambda E) \\
                              \end{array}
                                  \right] $$ is block pentadiagonal
if and only if  $L_{\sigma}(\lambda)$ is block pentadiagonal. Similar argument holds when $c_1 = 0$ and $i_1 = 1. $
\end{proof}

Let $\mathcal{O} = M_1M_3\cdots$ be the product of odd $M_i$ factors and $\mathcal{E} = M_2M_4\cdots$ be the product of the even $M_i$ factors, excluding $M_0$ and $M_m$. Then it is shown in \cite{TDM} that the product of $\mathcal{O}, \mathcal{E}$ and $M_0$ in any order gives a pentadiagonal Fiedler pencil of $P(\lam)$.

\begin{corollary}\label{pdfpc}
Let  $\mathbb{M}_{\sig_{1}}$ be the product of odd $\mathbb{M}_{i}$ factors and $\mathbb{M}_{\sig_{2}}$ be the product of even $\mathbb{M}_{i}$ factors, excluding $\mathbb{M}_0$ and $\mathbb{M}_m$. Then $\mathbb{L}_{\sig}(\lam) = \lam \mathbb{M}_{m} - \mathbb{M}_{\sig_{1}} \mathbb{M}_{0}\mathbb{M}_{\sig_{2}}$ and $\mathbb{L}_{\widetilde{\sig}}(\lam) = \lam \mathbb{M}_{m} - \mathbb{M}_{0}\mathbb{M}_{\sig_{2}}\mathbb{M}_{\sig_{1}} $ are block pentadiagonal pencils of $\mathcal{S}(\lam). $ However, the pencils $\mathbb{L}(\lam) = \lam \mathbb{M}_{m} - \mathbb{M}_{\sig_{2}}\mathbb{M}_{\sig_{1}} \mathbb{M}_{0}$ and $\mathbb{L}(\lam) = \lam \mathbb{M}_{m} -  \mathbb{M}_{0} \mathbb{M}_{\sig_{1}}\mathbb{M}_{\sig_{2}} $ are not block pentadiagonal.
\end{corollary}

\begin{proof}
Consider the pencil $\mathbb{L}_{\sig}(\lam) = \lam \mathbb{M}_{m} - \mathbb{M}_{\sig_{1}} \mathbb{M}_{0}\mathbb{M}_{\sig_{2}}$. Then CISS$(\sig) = (c_1,i_1, \ldots) = (0, 1, \ldots)$. Hence by Theorem \ref{fflptflr}, we have
 $$\mathbb{L}_{\sig}(\lam) = \left[
                             \begin{array}{c|c}
                                L_{\sigma}(\lambda) & e_{m-i_1}\otimes C \\
                              \hline
                            e_{m}^{T} \otimes B & (A - \lambda E) \\
                              \end{array}
                                  \right] = \left[
                             \begin{array}{c|c}
                                L_{\sigma}(\lambda) & e_{m-1}\otimes C \\
                              \hline
                            e_{m}^{T} \otimes B & (A - \lambda E) \\
                              \end{array}
                                  \right],$$ where $L_{\sig}(\lam) = \lam M_m - M_{\sig_1} M_0 M_{\sig_2}$. It is well known that the product of $M_{\sig_1}, M_{\sig_2}$, $M_0$ in any order gives a block pentadiagonal matrix (\cite{TDM}, Example $3.2$). Hence $L_{\sig}(\lam) = \lam M_m - M_{\sig}$ is block pentadiagonal and consequently $\mathbb{L}_{\sig}(\lam)$ is block pentadiagonal. For the pencil $\mathbb{L}_{\widetilde{\sig}}(\lam)$ we have CISS$(\widetilde{\sig}) = (c_1,i_1, \ldots) = (1, 1, \ldots)$. Hence by Theorem \ref{fflptflr},
$\mathbb{L}_{\widetilde{\sig}(\lam)}$ is block pentadiagonal. For $\mathbb{L}(\lam) = \lam \mathbb{M}_{m} - \mathbb{M}_{\sig_{2}}\mathbb{M}_{\sig_{1}} \mathbb{M}_{0}$, we have CISS$(\sig) = (c_1,i_1, \ldots) = (2, \ldots)$. By Theorem \ref{fflptflr}, $\mathbb{L}(\lam)$ has more than two subdiagonal blocks and hence is not block pentadiagonal. Similar argument holds for the pencil $\mathbb{L}(\lam) = \lam \mathbb{M}_{m} -  \mathbb{M}_{0} \mathbb{M}_{\sig_{1}}\mathbb{M}_{\sig_{2}}. $
\end{proof}

\begin{exam}\label{eopdfp} \em
Let  $G(\lam) := A_{6}\lam^{6} + \cdots + \lam A_1 + A_{0} + C(\lam E - A)^{-1}B$. Also, let $\sig^{-1}_{1} = (1, 3, 5)$ and $\sig^{-1}_{2} = (2, 4). $ We have
\begin{align*}
\mathbb{M}_{\sig_1} = {\scriptsize \left[
  \begin{array}{cccccc|c}
    -A_{5} & I_{n} &  &  &  &  &  \\
    I_{n} & 0 &  &  &  &  &  \\
     &  & -A_{3} & I_{n} &  &  &  \\
     &  & I_{n} & 0 &  &  &  \\
     &  &  &  & -A_{1} & I_{n} &  \\
     &  &  &  & I_{n}& 0 &  \\
    \hline
     &  &  &  &  &  & I_r \\
  \end{array}
\right], \,\, \mathbb{M}_{\sig_2} =
\left[
  \begin{array}{cccccc|c}
    I_{n} &  &  &  &  &  &  \\
     & -A_{4} & I_{n} &  &  &  &  \\
     & I_{n} & 0 &  &  &  &  \\
     &  &  & -A_{2} & I_{n} &  &  \\
     &  &  & I_{n} & 0 &  &  \\
     &  &  &  & & I_n&  \\
    \hline
     &  &  &  &  &  & I_r \\
  \end{array}
\right]}.
\end{align*}
Case I: Consider the pencil $\mathbb{L}_{\sigma}(\lambda) :=  \lambda \mathbb{M}_{6} - \mathbb{M}_{\sig_1} \mathbb{M}_{0}\mathbb{M}_{\sig_2}.$ Then we have
$$\mathbb{M}_{0} \mathbb{M}_{\sig_2} =
\left[
  \begin{array}{cccccc|c}
    I_{n} &  &  &  &  &  &  \\
     & -A_{4} & I_{n} &  &  &  &  \\
     & I_{n} & 0 &  &  &  &  \\
     &  &  & -A_{2} & I_{n} &  &  \\
     &  &  & I_{n} & 0 &  &  \\
     &  &  &  & & -A_{0}& -C \\
    \hline
     &  &  &  &  & -B & -A \\
  \end{array}
\right]$$
and hence
$$\mathbb{M}_{\sigma_1} \mathbb{M}_0\mathbb{M}_{\sigma_2}= \left[
    \begin{array}{cccccc|c}
      -A_{5} & -A_{4} & I_{n} & 0 & 0 & 0 & 0 \\
      I_{n} & 0 & 0 & 0 & 0 & 0 & 0 \\
      0 & -A_{3} & 0 & -A_{2} & I_{n} & 0 & 0 \\
      0 & I_{n} & 0 & 0 & 0 & 0 & 0 \\
      0 & 0 & 0 & -A_{1} & 0 & -A_{0} & -C \\
      0 & 0 & 0 & I_{n} & 0 & 0 & 0 \\
       \hline
      0 & 0 & 0 & 0 & 0 & -B & -A \\
    \end{array}
  \right]
$$
from which it follows that $\mathbb{L}_{\sigma}(\lam)$ is block pentadiagonal. Note that the number of consecutions at $0$ in $\mathbb{M}_{\sigma}$ is zero and the number of inversion at $0$ in $\mathbb{M}_{\sigma}$ is $1$.

Case II: Consider the pencil $\mathbb{L}_{\sig}(\lambda) :=  \lambda \mathbb{M}_{6} - \mathbb{M}_{0}\mathbb{M}_{\sig_{2}}\mathbb{M}_{\sig_{1}}. $ Then
$$\mathbb{M}_{0}\mathbb{M}_{\sig_{2}}\mathbb{M}_{\sig_{1}} = \left[
    \begin{array}{cccccc|c}
      -A_{5} & I_{n} & 0 & 0 & 0 & 0 & 0 \\
      -A_{4} & 0 & -A_{3} & I_{n} & 0 & 0 & 0 \\
      I_{n} & 0 & 0 & 0 & 0 & 0 & 0 \\
      0 & 0 & -A_{2} & 0 & -A_{1} & I_{n} & 0 \\
      0 & 0 & I_{n} & 0 & 0 & 0 & 0 \\
      0 & 0 & 0 & 0 & -A_{0} & 0 & -C \\
       \hline
      0 & 0 & 0 & 0 & -B & 0 & -A \\
    \end{array}
  \right]
$$
is block pentadiagonal. Thus $\mathbb{L}_{\sig}(\lam)$ is a block pentadiagonal pencil. Note that the number of consecutions at $0$ in $\mathbb{M}_{\sigma}$ is $1$, i.e., $c_1=1. $  \\

Case III: Consider the pencil $\mathbb{L}_{\sig}(\lambda) :=  \lambda \mathbb{M}_{6} - \mathbb{M}_{0}\mathbb{M}_{\sig_{1}}\mathbb{M}_{\sig_{2}}$.
Then  $\mathbb{L}_{\sig}(\lambda)$  is not pentadiagonal, since the number of consecutions at $0$ in $\mathbb{M}_{\sig}$ is $2$. Indeed, 
$$ \mathbb{M}_{0}\mathbb{M}_{\sig_{1}}\mathbb{M}_{\sig_{2}}=  \left[
  \begin{array}{cccccc|c}
    -A_{5} & -A_4 & I_n & 0 & 0 & 0 & 0 \\
    I_{n} & 0 & 0 & 0 & 0 & 0 & 0 \\
    0 & -A_{3} & 0 & -A_2 & I_{n} & 0 & 0 \\
    0 & I_{n} & 0 & 0 & 0 & 0 & 0 \\
    0 & 0 & 0 & -A_{1} & 0 & I_{n} & 0 \\
    0 & 0 & 0 & -A_0 & 0 & 0 & -C \\
    \hline
    0 & 0 & 0 & -B & 0 & 0 & -A  \\
  \end{array}
\right],
$$
which is not block pentadiagonal. Note however that $M_{\sig} := M_0M_{\sigma_1}M_{\sigma_2},$ which is the $(1, 1)$ block matrix in $\mathbb{M}_\sigma,$  is pentadiagonal. Similarly, $\mathbb{M}_{\sig} = \mathbb{M}_{\sig_{2}}\mathbb{M}_{\sig_{1}}\mathbb{M}_{0}$ is also not block pentadiagonal, since the number of inversions at $0$ in $\mathbb{M}_{\sig}$ is $2$. $\blacksquare$
\end{exam}

As in the case of Fiedler pencils of matrix polynomials, Fiedler pencils of Rosenbrock system polynomials can be constructed algorithmically.

\begin{theorem}
Let $\sigma : \{0, 1, \ldots, m-1\} \rightarrow \{1, 2, \ldots, m \}$ be a bijection. Let $\mathbb{L}_{\sigma}(\lam) = \lam \mathbb{M}_{m}- \mathbb{M}_{\sigma}$ be the Fiedler pencil of system matrix $\mathcal{S}(\lam)$ associated with $\sigma$. Consider the matrices $W_0, W_1, \ldots, W_{m-2}$ constructed by Algorithm~\ref{alg1}. Then we have $\mathbb{M}_{\sigma} = W_{m-2}$.

\begin{algorithm}[H]
\caption{Construction of $\mathbb{M}_{\sigma}$ for $\mathbb{L}_{\sigma}(\lam) := \lam \mathbb{M}_{m}- \mathbb{M}_{\sigma}$.}
\label{alg1}

\textbf{Input}: $\mathcal{S}(\lam) = \left[
                              \begin{array}{c|c}
                                \sum \limits_{i=0}^{m}\lam^{i}A_{i} & C \\
                                \hline
                                B & A -\lam E \\
                              \end{array}
                            \right]$ and a bijection $\sigma :\{0, 1, \ldots, m-1\} \rar \{1, 2, \ldots, m\}$. \\
\textbf{Output}: { $\mathbb{M}_{\sigma}$ }
\begin{algorithmic}

\If{$\sigma$ has a consecution at $0$}
    \State $W_0 := \left[
                    \begin{array}{cc|c}
                      -A_{1} & I_n & 0 \\
                      -A_0 & 0 & -C \\
                      \hline
                      -B & 0 & -A \\
                    \end{array}
                  \right]$

\Else
    \State $W_0 := \left[
                    \begin{array}{cc|c}
                      -A_{1} & -A_0 & - C \\
                       I_n & 0 & 0 \\
                      \hline
                      0 & -B & -A \\
                    \end{array}
                  \right]
    $
\EndIf

   \For{$i = 1:m-2$}

\If{$\sigma$ has a consecution at $i$}
    \State $W_i :=  \left[
       \begin{array}{cccc}
         -A_{i+1} & I_{n} & 0 & 0  \\
         W_{i-1}(:,1) & 0 & W_{i-1}(:,2:i+1) & W_{i-1}(:,i+2)   \\
       \end{array}
     \right]$

\Else
    \State $W_i := \left[
           \begin{array}{cc}
             -A_{i+1} & W_{i-1}(1,:)   \\
             I_{n} & 0    \\
             0 & W_{i-1}(2:i+1, :)  \\
             0 & W_{i-1}(i+2,:) \\
           \end{array}
         \right]
   $

\EndIf

\EndFor
\State $\mathbb{M}_{\sigma} := W_{m-2}$

\end{algorithmic}
\end{algorithm}

\end{theorem}

\begin{proof}
We prove the result by induction on the degree $m$. For $m = 2$ the proof is obvious, since for $m=2$ there are only three Fiedler factors and two possibility for $\mathbb{M}_{\sig}$, either $\mathbb{M}_{\sig} = \mathbb{M}_{0}\mathbb{M}_{1}$ if $\sig$ has a consecution at $0$ or $\mathbb{M}_{\sigma} = \mathbb{M}_{1}\mathbb{M}_{0}$ if $\sigma$ has an inversion at $0$. Now
$$\mathbb{M}_{0}\mathbb{M}_{1} = \left[
                    \begin{array}{cc|c}
                      -A_{1} & I_n & 0 \\
                      -A_0 & 0 & -C \\
                      \hline
                      -B & 0 & -A \\
                    \end{array}
                  \right] \text{  and  } \mathbb{M}_{1}\mathbb{M}_{0} = \left[
                    \begin{array}{cc|c}
                      -A_{1} & -A_0 & - C \\
                       I_n & 0 & 0 \\
                      \hline
                      0 & -B & -A \\
                    \end{array}
                  \right].$$ This proves the case for $m= 2. $
Suppose that the result is true for $m-1 \geq 2$. We show that it is true for $m$ and the bijection $\sig : \{0, 1, \ldots, m-1\} \rightarrow \{1, 2, \ldots, m\}. $ Note that the Fiedler matrices $\mathbb{M}_{i}$ associated with $\mathcal{S}(\lam) = \left[
                                                                                                                \begin{array}{c|c}
                                                                                                                  \sum\limits_{j = 0}^{m}\lam^{j}A_{j} & C \\
                                                                                                                  \hline
                                                                                                                  B & A - \lam E \\
                                                                                                                \end{array}
                                                                                                              \right]
$ satisfies
\begin{equation}
\mathbb{M}_{i} = \diag(I_{n}, \widetilde{M}_{i}), \text{  for  } i= 0:m-2,  \label{mialg}
\end{equation}
where $\widetilde{M}_{i}$ are the $(n(m-1)+r) \times (n(m-1)+r)$ Fiedler matrices associated with
$\widetilde{S}(\lam) = \left[
                         \begin{array}{c|c}
                           \sum\limits_{j = 0}^{m-1}\lam^{j}A_{j} & C \\
                           \hline
                           B & A - \lam E \\
                         \end{array}
                       \right]. $ \\

Case I. If $\sigma$ has a consecution at $m-2$, then using the commutativity relations of the matrices $\mathbb{M}_{i}$ we can write
$$\mathbb{M}_{\sig} = \mathbb{M}_{i_{0}}\cdots \mathbb{M}_{i_{m-2}} \mathbb{M}_{m-1}, $$ where $(i_{0}, i_{1}, \ldots, i_{m-2})$ is a permutation of
$(0, 1, \ldots, m-2). $ By (\ref{mialg}) we can write
\begin{equation}
\mathbb{M}_{\sigma} = \diag(I_{n}, \widetilde{M}_{\tau})\mathbb{M}_{m-1}, \label{indu}
\end{equation}
where $\tau : \{0, 1, \ldots, m-2\} \rightarrow \{1, \ldots, m-1\}$ is a bijection such that for $i= 0:m-3$, $\tau$ has a consecution (resp., inversion) at $i$ if and only if $\sigma$ has a consecution (resp., inversion) at $i$. So by the induction hypothesis, $\widetilde{M}_{\tau} = W_{m-3}$. Finally the product given in (\ref{indu}) is
$$
 {\small \mathbb{M}_{\sig}  = \left[
    \begin{array}{cccc}
      I_{n} & 0 & 0  & 0 \\
      0 & W_{m-3}(:, 1) & W_{m-3}(:, 2:m-1) & W_{m-3}(:, m) \\
    \end{array}
  \right]
 \left[
    \begin{array}{ccc|c}
      -A_{m-1} & I_{n} &  &  \\
      I_{n} & 0 &  &  \\
       &  & I_{(m-2)n} &  \\
      \hline
       &  &  & I_{r} \\
    \end{array}
  \right] } $$
$$ =  \left[
    \begin{array}{cccc}
      -A_{m-1} & I_{n} & 0 & 0 \\
      W_{m-3}(:, 1) & 0 & W_{m-3}(:, 2:m-1) & W_{m-3}(:, m)\\
    \end{array}
  \right], $$
which is equal to $W_{m-2}$ constructed when $\sig$ has a consecution at $m-2.$

Case II. If $\sig$ has an inversion at $m-2$ the proof is similar, but
$$\mathbb{M}_{\sig} = \mathbb{M}_{m-1}\mathbb{M}_{i_{0}}\cdots \mathbb{M}_{i_{m-2}}  = \mathbb{M}_{m-1}\diag(I_{n}, \widetilde{M}_{\tau}),$$
since $\mathbb{M}_{m-1}$ lies to the left of $\mathbb{M}_{m-2}. $
\end{proof}

\section{Fielder pencils are linearizations for system polynomials}
We now show that a Fiedler pencil  $\mathbb{L}_{\sig}(\lam)$ of the Rosenbrock system polynomial $ \mathcal{S}(\lam)$ is a {\em linearization} of  $\mathcal{S}(\lambda)$ in the sense of {\em system equivalence.} An $n$-by-$n$ matrix polynomial $X(\lam)$ is said to be unimodular if $\det(X(\lam))$ is a nonzero constant for all $\lam \in \C.$  Let $G_1(\lam)$ and $G_2(\lam)$ be  $n$-by-$n$ rational matrix functions. Then $ G_1(\lam)$ is said to be unimodularly equivalent to $G_2(\lam)$ (denoted as $G(\lam) \thicksim G_2(\lam)$)  if there are $n$-by-$n$ unimodular matrix polynomials $U(\lam)$ and $V(\lam)$ such that  $ U(\lam) G_1(\lam) V(\lam) = G_2(\lam)$ for all $ \lam$, see~\cite{rosenbrock70}.

\begin{definition}[System polynomial]
An $(n+r)\times (n+r)$ matrix polynomial $ \mathcal{X}(\lam)$ is said to be a  system polynomial of degree $m$ if it is of the form $$ \mathcal{X}(\lam) = \left[\begin{array}{c|c} X_{11}(\lam) &  X_{12}\\ \hline  X_{21} &  X_{22}(\lam) \end{array}\right],$$  where $X_{11}(\lam)  $ is an $n$-by-$n$ matrix polynomial of degree $m$, $X_{12}$ and $ X_{21}$ are constant matrices and $ X_{22}(\lam)$ is an $r$-by-$r$ matrix pencil. If, in addition, the pencil $X_{22}(\lam)$ is regular then $\mathcal{X}(\lam)$ is said to be a Rosenbrock system polynomial and, in such a case, the rational matrix function $$ T(\lam) := X_{11}(\lam) - X_{12} \left(X_{22}(\lam)\right)^{-1} X_{21}$$ is said to be the transfer function associated with $\mathcal{X}(\lam).$
\end{definition}

A system polynomial represents a linear time-invariant system in state-space form and is also referred to as a system matrix. Observe that if $ \mathcal{X}(\lam) $ is an $(n+r)\times (n+r)$ system polynomial then for any $n$-by-$n$ unimodular matrix polynomials $U(\lam)$ and $V(\lam),$  the matrix polynomial
$$ \mathcal{Z}(\lam) := \left[\begin{array}{c|c} U(\lam) &  0 \\  \hline 0 & I_r \end{array} \right] \mathcal{X}(\lam)  \left[ \begin{array}{c|c} V(\lam) & 0 \\ \hline  0&  I_r \end{array}\right]$$ need not be a system polynomial. We now introduce the notion of a system equivalence which is different from ``strict system equivalence" introduced by Rosenbrock~\cite{rosenbrock70}.

\begin{definition}[System equivalence] Let $ \mathcal{X}(\lam)$ and $\mathcal{Y}(\lam)$ be $(n+r)\times (n+r)$ system polynomials.  Then $ \mathcal{X}(\lam)$ is said to be system equivalent to  $\mathcal{Y}(\lam)$ (denoted as $ \mathcal{X}(\lam) \thicksim_{se} \mathcal{Y}(\lam)$) if there are $n$-by-$n$ unimodular matrix polynomials $ U(\lam)$ and $V(\lam)$ such that  for all   $\lam \in \C, $ we have $$ \left[\begin{array}{c|c} U(\lam) &  0 \\  \hline 0 & I_r \end{array} \right] \mathcal{X}(\lam)  \left[ \begin{array}{c|c} V(\lam) & 0 \\ \hline  0&  I_r \end{array}\right] = \mathcal{Y} (\lam).$$
\end{definition}

Suppose that $ \mathcal{X}(\lam)$ and  $\mathcal{Y}(\lam)$ are $(n+r)\times (n+r)$ system polynomials and that $ \mathcal{X}(\lam) \thicksim_{se} \mathcal{Y}(\lam).$
Then it follows that $X_{22}(\lam) = Y_{22}(\lam)$ for $\lam \in \C,$ where $X_{22}(\lam)$ and $Y_{22}(\lam)$ denote the $(2, 2)$ blocks of $\mathcal{X}(\lam)$ and $\mathcal{Y}(\lam),$  respectively.  Consequently, if $\mathcal{X}(\lam)$ is a Rosenbrock system polynomial then so is $\mathcal{Y}(\lam).$ Let $ T_{\mathcal{X}}(\lam)$ and $T_{\mathcal{Y}}(\lam),$ respectively, denote the transfer functions associated with $\mathcal{X}(\lam)$ and $\mathcal{Y}(\lam).$ Then it is easy to see that the transfer functions $T_{\mathcal{X}}(\lam)$ and $T_{\mathcal{Y}}(\lam)$ are unimodularly equivalent. Consequently, system equivalence preserves (finite) transmission zeros (and their partial multiplicities) of a linear time-invariant system in state-space form.

If $\mathcal{X}(\lam)$ is an $(n+r)\times (n+r)$ system polynomial then for any positive integer $p$ the $(n+p+r)\times (n+p+r)$ system polynomial given by   $$  I_p \oplus \mathcal{X}(\lam) :=\left[ \begin{array}{cc} I_p & 0 \\  0 & \mathcal{X}(\lam)\end{array}\right]  $$ is called an extended system polynomial~\cite{vardulakis}. This amounts to introducing $p$ additional zero output variables to the linear time-invariant system associated with $\mathcal{X}(\lam).$  We now introduce the notion of a system linearization for system polynomials which we refer to as {\em Rosenbrock linearization}.

\begin{definition}[Rosenbrock linearization] Let $ \mathcal{X}(\lam)$ be an $ (n+r)\times (n+r)$ system polynomial of degree $m >1.$ Then
an $(nm +r) \times (nm+r)$ system pencil $\mathbb{L}(\lambda)$  (system polynomial of degree $1$)  is said to be a Rosenbrock linearization of $\mathcal{X}(\lam)$ if $\mathbb{L}(\lambda)  \thicksim_{se} I_{(m-1)n}\oplus \mathcal{X}(\lam).$ In other words,  $\mathbb{L}(\lambda)$ is of the form $\mathbb{L}(\lambda) = \left[
                                                     \begin{array}{c|c}
                                                       L(\lam) & X \\
                                                       \hline
                                                       Y & A-\lam E \\
                                                     \end{array}
                                                   \right]$  and there are  $nm \times nm$ unimodular matrix polynomials $U(\lam)$ and $V(\lam)$ such that $$\left[
         \begin{array}{c|c}
           U(\lam) & 0 \\
           \hline
           0 & I_r \\
         \end{array}
       \right] \mathbb{L}(\lam) \left[
                                    \begin{array}{c|c}
                                     V(\lam) & 0 \\
                                     \hline
                                     0 & I_r \\
                                \end{array}
                               \right] = \left[
                                           \begin{array}{c|c}
                                             I_{(m-1)n} & 0 \\
                                             \hline
                                             0 &  \mathcal{X}(\lam) \\
                                           \end{array}
                                         \right],$$
                                         where $L(\lam)$ is an $nm \times nm$ matrix pencil, $A- \lam E$ is an $r$-by-$r$ matrix pencil and, $X$ and $Y$ are constant matrices. If, in addition, $U(\lam)$ and $V(\lam)$ are constant matrices then $\mathbb{L}(\lam)$ is said to be a strict Rosenbrock linearization of $\mathcal{X}(\lam)$.
\end{definition}

Observe that $\mathbb{L}(\lambda)$ is a {\em ``trimmed structured linearization"} of $\mathcal{X}(\lam).$ Indeed, $\mathbb{L}(\lambda)$ is a structured linearization  in the sense that $\mathbb{L}(\lambda)$ is itself a system matrix and hence is associated with a linear time-invariant system in SSF. Further,  $\mathbb{L}(\lambda)$ is a trimmed linearization of $\mathcal{X}(\lam)$ in the sense that $\mathbb{L}(\lambda)$ is an $(mn+r)\times (mn+r)$ pencil whereas a usual linearization of $\mathcal{X}(\lam)$ would result in a pencil of dimension $m(n+r) \times m(n+r).$ Furthermore, a byproduct of Rosenbrock linearization of $\mathcal{X}(\lam)$ is that it deflates some (all) of the infinite eigenvalues of $\mathcal{X}(\lam)$ when $\deg (\mathcal{X}(\lam)) \geq 2.$  Indeed,  suppose that $\mathcal{X}(\lam)$ is regular and that $\deg (\mathcal{X}(\lam)) \geq 2.$ Then obviously $\mathcal{X}(\lam)$ has eigenvalues at $\infty.$ On the other hand,  $\mathbb{L}(\lam)$ has an infinite eigenvalue if and only if at least one of the pencils  $A-\lam E$ and  $L(\lam)$ has an infinite eigenvalue. Consequently, $\mathbb{L}(\lam)$ does not have an infinite eigenvalue when the leading coefficients of the matrix polynomial  $X(\lam)$ and $ E$ are nonsingular even though the system polynomial $\mathcal{X}(\lam)$ has eigenvalues at $\infty.$ Thus  $\mathbb{L}(\lam)$ deflates all infinite eigenvalues of $\mathcal{X}(\lam)$ when $\deg(\mathcal{X}(\lam)) \geq 2$ and the leading coefficient of $X(\lam)$ is nonsingular.

Let $\mathcal{X}(\lam)$ be an $(n+r)\times (n+r)$ system polynomial of degree $m.$ Suppose that $\mathbb{L}(\lambda)$  is  a Rosenbrock linearization of $\mathcal{X}(\lam).$ Then  $\mathcal{X}(\lam)$ and $\mathbb{L}(\lambda)$ are of the form
\be \label{sysmat}  \mathcal{X}(\lam) = \left[\begin{array}{c|c} X(\lam) &  X_{12}\\ \hline  X_{21} &  Y-\lam Z \end{array}\right] \,\, \mbox{ and } \,\,
\mathbb{L}(\lambda) = \left[
                                                     \begin{array}{c|c}
                                                       L(\lam) & L_{12} \\
                                                       \hline
                                                       L_{21} & Y-\lam Z \\
                                                     \end{array}
                                                   \right], \ee  where $X(\lam)$ is an $n$-by-$n$ matrix polynomial of degree $m,$  $L(\lam)$ is an $mn\times mn$ matrix pencil, $Y- \lam Z$ is an $r$-by-$r$ matrix pencil and $ X_{12}, X_{21}, L_{12}, L_{21}$ are matrices of appropriate dimensions. It is easy to see that
$L(\lam)$  is unimodularly equivalent to $I_{(m-1)n}\oplus X(\lam),$ that is, $L(\lam)$ is a linearization~\cite{gohberg82,mackey2006vs} of the matrix polynomial $X(\lam). $ Thus $L(\lam)$ is a linearization of $X(\lam)$ when  $ \mathbb{L}(\lam)$ is a Rosenbrock linearization of $\mathcal{X}(\lam).$

On the other hand, suppose that  $L(\lam)$ is a linearization of $ X(\lam).$ Let $U(\lam)$ and $V(\lam)$ be $mn\times mn$ unimodular matrix polynomials such that $ U(\lam) L(\lam) V(\lam) = I_{(m-1)n}\oplus X(\lam)$ for all $ \lam \in \C.$ Then we have
 \be \label{slin}\left[
         \begin{array}{c|c}
           U(\lam) & 0 \\
           \hline
           0 & I_r \\
         \end{array}
       \right] \mathbb{L}(\lam) \left[
                                    \begin{array}{c|c}
                                     V(\lam) & 0 \\
                                     \hline
                                     0 & I_r \\
                                \end{array}
                               \right] = \left[
                                           \begin{array}{c|c}
                                             I_{(m-1)n} & 0 \\
                                             \hline
                                             0 &  \mathcal{X}(\lam) \\
                                           \end{array}
                                         \right], \ee where
$\mathbb{L}(\lam) := \left[\begin{array}{c|c}   L(\lam) &  U(\lam)^{-1} (e_m \otimes X_{12}) \\ \hline
(e_m^T \otimes X_{21}) V(\lam)^{-1} & Y- \lam Z \end{array}\right]. $ Since a system polynomial is in general not invariant under system equivalence, it follows that   $\mathbb{L}(\lam)$  is a Rosenbrock linearization of $\mathcal{X}(\lam)$ if and only if $U(\lam)^{-1} (e_m \otimes X_{12})$  and $ (e_m^T \otimes X_{21}) V(\lam)^{-1}$ are constant matrices independent of $\lam.$

\begin{corollary} Let $\mathcal{X}(\lam)$ and $\mathbb{L}(\lam)$ be as in (\ref{sysmat}). If $\mathbb{L}(\lam)$ is a Rosenbrock linearization of $\mathcal{X}(\lam)$ then $ L(\lam)$ is a linearization of $X(\lam).$  Conversely, if $ L(\lam)$ is a linearization of $X(\lam)$ then $\mathbb{L}(\lam)$ is a Rosenbrock linearization of $\mathcal{X}(\lam)$ if and only if $U(\lam)^{-1} (e_m \otimes X_{12}) $ and $ (e_m^T \otimes X_{21}) V(\lam)^{-1} $ are constant matrices (independent of $\lam$) and are equal to $L_{12}$ and $L_{21},$ respectively.
\end{corollary}

So, the moot question is: Are $U(\lam)^{-1} (e_m \otimes X_{12}) $ and $ (e_m^T \otimes X_{21}) V(\lam)^{-1} $ constant matrix polynomials when $L(\lam)$ is a linearization of $X(\lam)?$ It turns out that it is indeed the case when $L(\lam)$ is a Fiedler pencil of $X(\lam)$ and that
the Fiedler pencils in Theorem~\ref{fflptflr} are Rosenbrock linearizations of $\mathcal{X}(\lam).$ Computation of the last block  column and row of $ U(\lam)^{-1}$ and $V(\lam)^{-1},$ respectively, is a laborious task  and involves long calculations. We therefore adapt the method of proof presented in~\cite{TDM} for matrix polynomials to the case of system polynomials and show that Fiedler pencils of system polynomials are indeed Rosenbrock linearizations.

\begin{definition}[Horner shift, \cite{TDM}]
Let $P(\lambda) = A_{0} + \lambda A_{1} + \cdots + \lambda^{m}A_{m}$ be a matrix polynomial of degree $m$. For $k = 0, \ldots m$, the degree $k$ Horner shift of $P(\lambda)$ is the matrix polynomial $P_{k}(\lambda) := A_{m-k} + \lambda A_{m-k+1} + \cdots + \lambda^{k}A_{m}$. These Horner shifts satisfy the following:
\begin{align}
& P_{0}(\lambda) = A_{m}, \nonumber \\
& P_{k+1}(\lambda) = \lambda P_{k}(\lambda) + A_{m-k-1} , \mbox{  for } 0\leq k\leq m-1, \\
& P_{m}(\lambda) = P(\lambda). \nonumber \label{hs}
\end{align}
\end{definition}

Let $H := (H_{ij})$ be a block $m \times n$ matrix with $p\times q$  blocks $H_{ij}.$ The {\em block transpose} of $H,$ denoted by $H^{\mathcal{B}},$ is the block $n \times m$ matrix  with $p \times q$ blocks defined by $(H^{\mathcal{B}})_{ij} := H_{ji},$ see~\cite{TDM}.  Consider the auxiliary matrices associated with a matrix polynomial given below.

\begin{definition}[\cite{TDM}, Definition~4.2] \label{qrtdmfmp}
Let $P(\lambda) = \sum\limits_{i = 0}^{m}\lambda^{i} A_{i}$ be an $n \times n$ matrix polynomial, and let $P_{i}(\lambda)$ be the degree $i$ Horner shift of $P(\lambda)$. For $1 \leq i \leq m-1$, define the following $nm \times nm$ matrix polynomials:
{\small $$
 Q_{i}(\lambda) := \left[
                    \begin{array}{cccc}
                      I_{(i-1)n} &  &  &  \\
                       & I_{n} & \lambda I_{n} &  \\
                       & 0_{n} & I_{n} &  \\
                       &  &  & I_{(m-i-1)n} \\
                    \end{array}
                  \right], \,\, R_{i}(\lambda) := \left[
      \begin{array}{cccc}
        I_{(i-1)n} &  &  &  \\
          & 0_{n} & I_{n} &  \\
         & I_{n} & P_{i}(\lambda) &  \\
          &  &  & I_{(m-i-1)n} \\
      \end{array}
    \right], $$

$$ T_{i}(\lam) := \left[
      \begin{array}{cccc}
        0_{(i-1)n} &  &  &  \\
          & 0_{n} & \lam P_{i-1}(\lam) &  \\
         & \lam I_{n} & \lam^{2}P_{i-1}(\lam) &  \\
          &  &  & 0_{(m-i-1)n} \\
      \end{array}
    \right], $$
$$ D_{i}(\lambda) := \left[
      \begin{array}{cccc}
        0_{(i-1)n} &  &  &  \\
          & P_{i-1}(\lam) & 0_{n}  &  \\
         & 0_{n} & I_{n} &  \\
          &  &  & I_{(m-i-1)n} \\
      \end{array}
    \right], $$}
and  $D_{m}(\lam) := \diag \left[0_{(m-1)n},  P_{m-1}(\lam)\right].
$
For simplicity, we often write $Q_{i}, R_{i}, T_{i}, D_{i}$ in place of $Q_{i}(\lam), R_{i}(\lam), T_{i}(\lam), D_{i}(\lam)$. Note that $D_{1}(\lambda) = M_{m}$, and $Q_{i}(\lam)$, $R_{i}(\lam)$ are unimodular for all $i = 1, \ldots, m-1$. Also note that $R_{i}^{\mathcal{B}}(\lam) = R_{i}(\lambda).$
\end{definition}

The auxiliary matrices satisfy the following relations.

\begin{lemma}[\cite{TDM}, Lemma~4.3]
Let $Q_{i}, R_{i}, T_{i}, D_{i}$ be as in Definition \ref{qrtdmfmp} and $M_i$'s be Fiedler matrices associated with $P(\lam).$ Then the following relations hold for $i = 1, \ldots, m-1$.
\begin{itemize}

\item[(a)] $Q_{i}^{\mathcal{B}}(\lambda D_{i})R_{i} = \lambda D_{i+1} + T_{i}$, and $Q_{i}^{\mathcal{B}}(M_{m-(i+1)}M_{m-i})R_{i} = M_{m-(i+1)} + T_{i}$.

\item[(b)] $R_{i}^{\mathcal{B}}(\lambda D_{i})Q_{i} = \lambda D_{i+1} + T_{i}^{\mathcal{B}}$, and $R_{i}^{\mathcal{B}}(M_{m-i}M_{m-(i+1)})Q_{i} = M_{m-(i+1)} + T_{i}^{\mathcal{B}}$.

\item[(c)] $T_{i}M_{j} = M_{j}T_{i} = T_{i}$ and $T_{i}^{\mathcal{B}}M_{j} = M_{j}T_{i}^{\mathcal{B}} = T_{i}^{\mathcal{B}}$ for all $j \leq m-i-2$.
\end{itemize}
\end{lemma}

We extend these results to the case of system polynomials. For this purpose, we introduce ``block transpose" of a system matrix  as follows.

\begin{definition}[Block transpose]\label{btfrmf} Let $\mathcal{A} $ be an $(mn+r)\times (mn+r)$ system matrix given by
 $$\mathcal{A} := \left[
                 \begin{array}{c|c}
                   A & e_i \otimes X \\
                   \hline
                   e_j^{T} \otimes Y & Z \\
                 \end{array}
               \right], $$   where $A := [A_{ij}]$ is an $m \times m$ block matrix with $A_{ij} \in \C^{n \times n}$, $X \in \C^{n\times r}, Y \in \C^{r\times n},$ $Z \in \C^{r \times r}$ and $e_k$ is the $k$-th column of $I_m.$  The {\em block transpose} of $\mathcal{A},$ denoted by $\mathcal{A}^{\mathbb{B}},$ is defined by $$\mathcal{A}^{\mathbb{B}} := \left[
                 \begin{array}{c|c}
                   A^{\mathcal{B}} & e_j \otimes X \\
                   \hline
                   e_i^{T} \otimes Y & Z \\
                 \end{array}
               \right], $$ where $A^{\mathcal{B}}$ is the block transpose of $A.$
\end{definition}

Recall that $C_P(\lam)$ is the first companion form of the matrix polynomial $P(\lam).$ It is well known that the second companion form $C_2(\lam)$ of $ P(\lam) $ is the block transpose of $C_P(\lam),$ that is, $ C_2(\lam) = {C_P(\lam)}^{\mathcal{B}},$ see~\cite{TDM}. It turns out that the same is true for the second companion form $\mathcal{C}_2(\lam)$ of the system polynomial $\mathcal{S}(\lam).$ Indeed, in view of Theorem~\ref{fflptflr}, Proposition~\ref{compan}, and (\ref{scfor}), we have $\mathcal{C}_2(\lam) = \mathcal{C}(\lam)^{\mathbb{B}}.$

We now define auxiliary system polynomials as follows.

\begin{definition}[Auxiliary system polynomials] \label{amr}
Let $Q_{i}(\lambda), R_{i}(\lambda), T_{i}(\lambda)$, and $D_{i}(\lambda)$ be as in Definition \ref{qrtdmfmp}.
For $i = 1, \ldots, m-1$, define $(nm+r) \times (nm+r)$ system polynomials:
\beano
\mathcal{Q}_{i}(\lambda) &:=& \left[
                                \begin{array}{c|c}
                                  Q_{i}(\lambda) & 0 \\
                                  \hline
                                  0 & I_{r} \\
                                \end{array}
                              \right], \indent \mathcal{R}_{i}(\lambda) := \left[
                                                                             \begin{array}{c|c}
                                                                            R_{i}(\lambda) & 0 \\
                                                                               \hline
                                                                             0 & I_{r} \\
                                                                             \end{array}
                                                                            \right],  \\
\mathcal{T}_{i}(\lambda) &:=& \left[
                                \begin{array}{c|c}
                                  T_{i}(\lambda) & 0 \\
                                  \hline
                                  0 & 0_{r} \\
                                \end{array}
                              \right], \indent \mathcal{D}_{i}(\lambda) := \left[
                                                                             \begin{array}{c|c}
                                                                            D_{i}(\lambda) & 0 \\
                                                                               \hline
                                                                             0 & -E \\
                                                                             \end{array}
                                                                            \right],  \\
\mbox{ and }\,\, \mathcal{D}_{m}(\lambda) & :=& \left[
                                  \begin{array}{c|c}
                                    D_{m} & 0 \\
                                    \hline
                                    0 & -E \\
                                  \end{array}
                                \right].  \eeano
Note that $\mathcal{D}_{1}(\lambda) = \left[
                                          \begin{array}{c|c}
                                            D_{1}(\lambda) & 0 \\
                                            \hline
                                            0 & -E \\
                                          \end{array}
                                        \right] = \left[
                                                    \begin{array}{c|c}
                                                      M_{m} & 0 \\
                                                    \hline
                                                      0 & -E \\
                                                    \end{array}
                                                  \right] = \mathbb{M}_{m}$ and that $\mathcal{Q}_{i}(\lambda)$ and $\mathcal{R}_{i}(\lambda)$ are unimodular matrix polynomials for $i = 1, \ldots, m-1$. Also, note that $\mathcal{R}_{i}^{\mathbb{B}}(\lambda) = \mathcal{R}_{i}(\lambda)$ for $i = 1, \ldots, m-1$.

\end{definition}

The auxiliary system polynomials satisfy the following relations.

\begin{lemma}
Let $\mathcal{Q}_{i}, \mathcal{R}_{i}, \mathcal{T}_{i}, \mathcal{D}_{i}$ be the system polynomials given in Definition \ref{amr} and $\mathbb{M}_i$'s be Fiedler matrices associated with $\mathcal{S}(\lam).$ Then the following system equivalence relations hold for $i = 1, \ldots, m-1$.
\begin{itemize}

\item[(a)] $\mathcal{Q}_{i}^{\mathbb{B}}(\lambda \mathcal{D}_{i})\mathcal{R}_{i} = \lambda \mathcal{D}_{i+1} + \mathcal{T}_{i}$, and $\mathcal{Q}_{i}^{\mathbb{B}}(\mathbb{M}_{m-(i+1)}\mathbb{M}_{m-i})\mathcal{R}_{i} = \mathbb{M}_{m-(i+1)} + \mathcal{T}_{i}$.

\item[(b)] $\mathcal{R}_{i}^{\mathbb{B}}(\lambda \mathcal{D}_{i})\mathcal{Q}_{i} = \lambda \mathcal{D}_{i+1} + \mathcal{T}_{i}^{\mathbb{B}}$, and $\mathcal{R}_{i}^{\mathbb{B}}(\mathbb{M}_{m-i}\mathbb{M}_{m-(i+1)})\mathcal{Q}_{i} = \mathbb{M}_{m-(i+1)} + \mathcal{T}_{i}^{\mathbb{B}}$.

\item[(c)] $\mathcal{T}_{i}\mathbb{M}_{j} = \mathbb{M}_{j}\mathcal{T}_{i} = \mathcal{T}_{i}$ and $\mathcal{T}_{i}^{\mathbb{B}}\mathbb{M}_{j} = \mathbb{M}_{j}\mathcal{T}_{i}^{\mathbb{B}} = \mathcal{T}_{i}^{\mathbb{B}}$ for all $j \leq m-i-2$.
\end{itemize}
\end{lemma}

\begin{proof}
 (a) We have
\beano
\mathcal{Q}_{i}^{\mathbb{B}}(\lambda \mathcal{D}_{i})\mathcal{R}_{i}
& =&  \left[
     \begin{array}{c|c}
       Q_{i}^{\mathcal{B}} & 0 \\
       \hline
       0 & I_{r} \\
     \end{array}
   \right] \left[
            \begin{array}{c|c}
            \lam D_{i} & 0 \\
            \hline
            0 & -\lam E \\
             \end{array}
               \right] \left[
                         \begin{array}{c|c}
                          R_{i} & 0 \\
                          \hline
                          0 & I_{r} \\
                          \end{array}
                            \right]
 =\left[
      \begin{array}{c|c}
       Q_{i}^{\mathcal{B}} \lambda D_{i}R_{i} & 0 \\
       \hline
      0 & -\lam E \\
        \end{array}
        \right]\\ & =&  \left[
                     \begin{array}{c|c}
                    \lam D_{i+1}+T_{i} & 0 \\
                     \hline
                     0 & -\lam E \\
                      \end{array}
                         \right] = \lambda \mathcal{D}_{i+1} + \mathcal{T}_{i}
\eeano
and
$$
\mathcal{Q}_{i}^{\mathbb{B}}(\mathbb{M}_{m-(i+1)}\mathbb{M}_{m-i})\mathcal{R}_{i} =
\left[
  \begin{array}{c|c}
    Q_{i}^{\mathcal{B}} & 0 \\
    \hline
    0 & I_{r} \\
  \end{array}
\right]\left[
         \begin{array}{c|c}
           M_{m-(i+1)}& 0 \\
           \hline
           0 & I_{r} \\
         \end{array}
       \right]\left[
                \begin{array}{c|c}
                  M_{m-i} & 0 \\
                  \hline
                  0 & I_{r} \\
                \end{array}
              \right]\left[
                       \begin{array}{c|c}
                         R_{i} & 0 \\
                         \hline
                         0 & I_{r} \\
                       \end{array}
                     \right] $$
$ = \left[
      \begin{array}{c|c}
    Q_{i}^{\mathcal{B}} M_{m-(i+1)} M_{m-i} R_{i} & 0 \\
        \hline
        0 & I_{r} \\
      \end{array}
    \right] = \left[
              \begin{array}{c|c}
                M_{m -(i+1)} + T_{i} & 0 \\
                \hline
                0 & I_{r} \\
              \end{array}
            \right]
 = \mathbb{M}_{m -(i+1)} + \mathcal{T}_{i}.$
For $i = m-1$, we have
$
\mathcal{Q}_{m-1}^{\mathbb{B}}(\mathbb{M}_{0}\mathbb{M}_{1})\mathcal{R}_{m-1} $
\beano &=&
\left[
  \begin{array}{c|c}
    Q_{m-1}^{\mathcal{B}} & 0 \\
    \hline
    0 & I_{r} \\
  \end{array}
\right]\left[
         \begin{array}{c|c}
           M_{0}& -e_{m}\otimes C \\
           \hline
           -e_{m}^{T}\otimes B & -A \\
         \end{array}
       \right]\left[
                \begin{array}{c|c}
                  M_{1} & 0 \\
                  \hline
                  0 & I_{r} \\
                \end{array}
              \right]\left[
                       \begin{array}{c|c}
                         R_{m-1} & 0 \\
                         \hline
                         0 & I_{r} \\
                       \end{array}
                     \right] \\
& = &\left[
  \begin{array}{c|c}
    Q_{m-1}^{\mathcal{B}} & 0 \\
    \hline
    0 & I_{r} \\
  \end{array}
\right]\left[
         \begin{array}{c|c}
           M_{0}M_{1}R_{m-1} & -e_{m}\otimes C \\
           \hline
           -e_{m}^{T}\otimes B M_{1} R_{m-1} & -A \\
         \end{array}
       \right] \\
& =& \left[
  \begin{array}{c|c}
   Q_{m-1}^{\mathcal{B}}M_{0}M_{1}R_{m-1}  & Q_{m-1}^{\mathcal{B}} (-e_{m}\otimes C) \\
    \hline
   -e_{m}^{T}\otimes B M_{1} R_{m-1} & -A \\
  \end{array}
\right]
= \left[
      \begin{array}{c|c}
      M_{0} + T_{m-1} & -e_{m}\otimes C \\
       \hline
       -e_{m}^{T}\otimes B & -A \\
         \end{array}
         \right] \\  &=& \mathbb{M}_{0} + \mathcal{T}_{m-1},
         \eeano
since $\left[
         \begin{array}{ccc}
           I_{(m-2)n} &  &  \\
            & I_{n} & \lambda I_{n} \\
            & 0 & I_{n} \\
         \end{array}
       \right] \left[
                 \begin{array}{c}
                   0 \\
                   \vdots \\
                   -C \\
                 \end{array}
               \right] = -e_{m}\otimes C
$ and
\beano
-e_{m}^{T}\otimes B M_{1} R_{m-1} &=& -e_{m}^{T}\otimes B\left[
                                        \begin{array}{ccc}
                                          I_{(m-2)n} &  &  \\
                                           & -A_{1} & I_{n} \\
                                           & I_{n} & 0 \\
                                        \end{array}
                                      \right]\left[
                                               \begin{array}{ccc}
                                                 I_{(m-2)n} &  &  \\
                                                  & 0 & I_{n} \\
                                                  & I_{n} & P_{m-1}(\lambda) \\
                                               \end{array}
                                             \right]\\
& =& \left[
      \begin{array}{ccc}
        0 & \cdots & -B \\
      \end{array}
    \right] \left[
              \begin{array}{ccc}
                I_{(m-2)n} &  &  \\
                 & I_{n} & -A_{1}+ P_{m-1}(\lambda) \\
                 & 0 & I_{n} \\
              \end{array}
            \right] = -e_{m}^{T}\otimes B.
\eeano

 (b) Now
\beano
\mathcal{R}_{i}^{\mathbb{B}}(\lambda \mathcal{D}_{i})\mathcal{Q}_{i} & =&
\left[
  \begin{array}{c|c}
    R_{i} & 0 \\
    \hline
    0 & I_{r} \\
  \end{array}
\right]\left[
         \begin{array}{c|c}
           \lambda D_{i} & 0 \\
           \hline
           0 & -\lambda E \\
         \end{array}
       \right]\left[
                \begin{array}{c|c}
                  Q_{i} & 0 \\
                  \hline
                  0 & I_{r} \\
                \end{array}
              \right]
 = \left[
      \begin{array}{c|c}
        R_{i} \lambda D_{i}Q_{i} & 0 \\
        \hline
        0 & -\lambda E \\
      \end{array}
    \right] \\ & =& \left[
                \begin{array}{c|c}
                  \lambda D_{i+1} + T_{i}^{\mathcal{B}} & 0 \\
                  \hline
                  0 & -\lambda E \\
                \end{array}
              \right]
= \lambda \mathcal{D}_{i+1} + \mathcal{T}_{i}^{\mathbb{B}}.
\eeano
Similarly $\mathcal{R}_{i}^{\mathbb{B}}(\mathbb{M}_{m-i}\mathbb{M}_{m-(i+1)})\mathcal{Q}_{i} = \mathbb{M}_{m-(i+1)} + \mathcal{T}_{i}^{\mathbb{B}}$.

 (c) Finally, we have
 \beano
\mathcal{T}_{i} \mathbb{M}_{j} &=& \left[
                                    \begin{array}{c|c}
                                      T_{i} & 0 \\
                                      \hline
                                      0 & 0 \\
                                    \end{array}
                                  \right]\left[
                                           \begin{array}{c|c}
                                             M_{j} & 0 \\
                                             \hline
                                             0 & I_{r} \\
                                           \end{array}
                                         \right] = \left[
                                                     \begin{array}{c|c}
                                                       T_{i}M_{j} & 0 \\
                                                       \hline
                                                       0 & 0 \\
                                                     \end{array}
                                                   \right]  = \left[
                                                               \begin{array}{c|c}
                                                                M_{j} T_{i} & 0 \\
                                                                 \hline
                                                                 0 & 0 \\
                                                               \end{array}
                                                             \right] \\ &=& \mathbb{M}_{j} \mathcal{T}_{i}
 =    \left[
         \begin{array}{c|c}
           T_{i} & 0 \\
           \hline
           0 & 0 \\
         \end{array}
       \right] = \mathcal{T}_{i}.
\eeano
Similarly $\mathcal{T}_{i}^{\mathbb{B}} \mathbb{M}_{j} = \mathbb{M}_{j} \mathcal{T}_{i}^{\mathbb{B}} = \mathcal{T}_{i}^{\mathbb{B}}$. For $i = m-2$ we have
\beano
\mathcal{T}_{m-2}^{\mathbb{B}} \mathbb{M}_{0} &=& \left[
                                                    \begin{array}{c|c}
                                                      T_{m-2}^{\mathcal{B}} & 0 \\
                                                      \hline
                                                      0 & 0 \\
                                                    \end{array}
                                                  \right]\left[
                                                           \begin{array}{c|c}
                                                             M_{0} & -e_{m}\otimes C \\
                                                             \hline
                                                             -e_{m}^{T}\otimes B & -A \\
                                                           \end{array}
                                                         \right]
= \left[
       \begin{array}{c|c}
         T_{m-2}^{\mathcal{B}}M_{0} & T_{m-2}^{\mathcal{B}}(-e_{m}\otimes C) \\
         \hline
         0 & 0 \\
       \end{array}
     \right] \\
 &=& \left[
   \begin{array}{c|c}
    T_{m-2}^{\mathcal{B}}M_{0} & 0 \\
    \hline
    0 & 0 \\
      \end{array}
        \right] = \left[
                    \begin{array}{c|c}
                      T_{m-2}^{\mathcal{B}} & 0 \\
                      \hline
                      0 & 0 \\
                    \end{array}
                  \right] = \mathcal{T}_{m-2}^{\mathbb{B}},
\eeano
since $T_{m-2}^{\mathcal{B}}(-e_{m}\otimes C) = \left[
         \begin{array}{cccc}
           0_{(m-3)n} &  &  &  \\
            & 0 & \lambda I  &  \\
            & \lambda P_{m-3} & \lambda^{2}P_{m-3} &  \\
            &  &  & 0_{n} \\
         \end{array}
       \right] \left[
                 \begin{array}{c}
                   0 \\
                   \vdots \\
                   0 \\
                   -C \\
                 \end{array}
               \right] = 0.
$
This completes the proof.
\end{proof}

As in the case of matrix polynomial~\cite{TDM},  we now define a family of system pencils which will form the intermediate steps in the system equivalence  transformation of a Fiedler pencil $\mathbb{L}_{\sigma}(\lam)$  into $I_{(m-1)n} \oplus \mathcal{S}(\lam).$ Note that a Fiedler pencil of $\mathcal{S}(\lam)$ is a system pencil.

\begin{definition}\label{syspendef}
Let $\mathbb{L}_{\sigma}(\lambda) = \lambda \mathbb{M}_{m} - \mathbb{M}_{\sigma}$ be the Fiedler pencil of $\mathcal{S}(\lambda)$ associated with a bijection $\sigma$. For $j = 1, 2, \ldots, m$, define
$$\mathbb{M}_{\sigma}^{(j)} := \prod_{\sigma^{-1}(i)\leq m-j}\mathbb{M}_{\sigma^{-1}(i)}, $$ where the factors $\mathbb{M}_{\sigma^{-1}(i)}$ are in the same relative order as they are in $\mathbb{M}_{\sigma}$. Note that $\mathbb{M}_{\sigma}^{(1)} = \prod_{\sigma^{-1}(i)\leq m-1}\mathbb{M}_{\sigma^{-1}(i)} = \mathbb{M}_{\sigma}$  and that $\mathbb{M}_{\sigma}^{(m)} = \mathbb{M}_{0}$. Also for $j = 1, 2, \ldots, m$ define the $(nm+r)\times (nm+r)$ system pencils
$\mathbb{L}_{\sigma}^{(j)}(\lambda) := \lambda \mathcal{D}_{j}(\lambda) - \mathbb{M}_{\sigma}^{(j)}. $
 Observe that $\mathbb{L}_{\sigma}^{(1)}(\lambda) = \lambda \mathcal{D}_{1} - \mathbb{M}_{\sigma}^{(1)} = \lambda \mathbb{M}_{m} - \mathbb{M}_{\sigma} = \mathbb{L}_{\sigma}$ and that $$
\mathbb{L}_{\sigma}^{(m)}(\lambda) = \lambda \mathcal{D}_{m} - \mathbb{M}_{\sigma}^{(m)}
= \lambda \left[
            \begin{array}{c|c}
              D_{m} & 0 \\
             \hline
             0 & -E \\
            \end{array}
          \right] - \mathbb{M}_{0} =  \left[
            \begin{array}{c|c}
              -I_{(m-1)n} & 0 \\
             \hline
             0 & \mathcal{S}(\lam) \\
            \end{array}
          \right]. $$

\end{definition}

The next result shows that  $\mathbb{L}_{\sigma}^{(i)}(\lambda) \thicksim_{se} \mathbb{L}_{\sigma}^{(i+1)}(\lambda)$ for $i=1, 2, \ldots, m-1.$

\begin{lemma} \label{lietli+1} We have
$\mathbb{L}_{\sigma}^{(i)}(\lambda)\thicksim_{se} \mathbb{L}_{\sigma}^{(i+1)}(\lambda)$ for $i = 1, 2, \ldots, m-1$. More precisely,  if $\mathcal{Q}_{i}$ and $\mathcal{R}_{i}$ are the system polynomials given in Definition~\ref{amr}, then
$$\mathbb{L}_{\sigma}^{(i+1)}(\lambda)=
\begin{cases}
\mathcal{Q}_{i}^{\mathbb{B}} \mathbb{L}_{\sigma}^{(i)}(\lambda) \mathcal{R}_{i}, &  \text{ if } \sigma \text{ has a consecution at } m-i-1 \\
\mathcal{R}_{i}^{\mathbb{B}} \mathbb{L}_{\sigma}^{(i)}\mathcal{Q}_{i},  & \text{ if } \sigma \text{ has an inversion at } m-i-1.
\end{cases}$$
\end{lemma}

\begin{proof} The proof is exactly the same as that of Lemma~4.5 in \cite{TDM}.
\end{proof}

It is now immediate that  a Fiedler pencil is a Rosenbrock linearization of $\mathcal{S}(\lam).$

\begin{theorem}[Rosenbrock linearization] Let $\mathcal{S}(\lam)$ be an $(n+r)\times (n+r)$ system polynomial (regular or singular) of degree $m.$
 Then a Fiedler pencil $\mathbb{L}_{\sigma}(\lambda)$ of the system polynomial $\mathcal{S}(\lam)$ is a Rosenbrock linearization of $\mathcal{S}(\lam)$.
\end{theorem}

\begin{proof} By Lemma~\ref{lietli+1}, we have $m-1$ system equivalences
\begin{equation}
\mathbb{L}_{\sigma}(\lambda) = \mathbb{L}_{\sigma}^{(1)}(\lambda) \thicksim_{se} \mathbb{L}_{\sigma}^{(2)}(\lambda) \thicksim_{se}\cdots \thicksim_{se} \mathbb{L}_{\sigma}^{(m)}(\lambda) = \left[
                                          \begin{array}{c|c}
                                            -I_{(m-1)n}   &  0 \\
                                              \hline
                                             0 & \mathcal{S}(\lam)
                                                \end{array}
                                                  \right], \label{uefl}
\end{equation}
where $\mathbb{L}_{\sigma}^{(i)}(\lambda)$ is as in Lemma~\ref{lietli+1}.  This shows that $\mathbb{L}_{\sigma}(\lambda) \thicksim_{se} I_{(m-1)n} \oplus \mathcal{S}(\lam).$ \end{proof}

The  $m-1$ system equivalences in (\ref{uefl}) provide the desired system equivalence that transforms $\mathbb{L}_{\sigma}(\lambda)$ to the extended system polynomial $I_{(m-1)n}\oplus \mathcal{S}(\lam)$  and can be constructed as follows.

\begin{corollary}\label{eouavflor}
Let $\mathbb{L}_{\sigma}(\lambda)$ be the Fiedler pencil of $\mathcal{S}(\lambda)$ associated with a bijection $\sigma$, and $\mathcal{Q}_{i}, \mathcal{R}_{i}$ for $i = 1, 2, \ldots m-1$, be as in Definition \ref{amr}. Then
\begin{equation}
\mathcal{U}(\lambda) \mathbb{L}_{\sigma}(\lambda)\mathcal{V}(\lambda) = \left[
                                                           \begin{array}{c|c}
                                                              -I_{(m-1)n} &  0  \\
                                                               \hline
                                                               0 & \mathcal{S}(\lam) \\
                                                                 \end{array}
                                                                  \right], \label{lr}
\end{equation}
where $\mathcal{U}(\lambda)$ and $\mathcal{V}(\lambda)$ are $(nm+r)\times (nm+r)$ unimodular system polynomials given by
$$\mathcal{U}(\lambda) := \mathcal{U}_{0}\mathcal{U}_{1}\cdots \mathcal{U}_{m-3}\mathcal{U}_{m-2}, \mbox{  with  } \mathcal{U}_{i} =
\begin{cases}
\mathcal{Q}_{m-(i+1)}^{\mathbb{B}}, & \text{if } \sigma \text{ has a consecution at } i, \\
\mathcal{R}_{m-(i+1)}^{\mathbb{B}}, & \text{if } \sigma \text{ has an inversion at } i,
\end{cases}
$$
$$\mathcal{V}(\lambda) := \mathcal{V}_{m-2}\mathcal{V}_{m-3}\cdots \mathcal{V}_{1}\mathcal{V}_{0}, \mbox{  with  } \mathcal{V}_{i} =
\begin{cases}
\mathcal{R}_{m-(i+1)}, & \text{if } \sigma \text{ has a consecution at } i, \\
\mathcal{Q}_{m-(i+1)}, & \text{if } \sigma \text{ has an inversion at } i.
\end{cases}
$$
\end{corollary}

The indexing of $\mathcal{U}_i$ and $\mathcal{V}_i$ factors in $\mathcal{U}(\lam)$ and $\mathcal{V}(\lam),$ respectively,  in Corollary~\ref{eouavflor} has been chosen for simplification of notation and has no other special significance.

\begin{exam} For $m=4$ and  $\mathbb{L}_{\sigma}(\lambda) = \lambda \mathbb{M}_{4} - \mathbb{M}_{2}\mathbb{M}_{0}\mathbb{M}_{1}\mathbb{M}_{3}$, by Corollary~\ref{eouavflor}, we have $(\mathcal{Q}_{3}^{\mathbb{B}}\mathcal{R}_{2}\mathcal{Q}_{1}^{\mathbb{B}})\mathbb{L}_{\sigma}(\lambda)(\mathcal{R}_{1}
\mathcal{Q}_{2}\mathcal{R}_{3}) = -I_{3n} \oplus \mathcal{S}(\lam).$  $\blacksquare$
\end{exam}

Observe that the invariant zeros of an  LTI system in SSF such as $\Sigma$ in (\ref{rassf}) can be computed in two steps. First, construct a Fiedler pencil $\L_{\sigma}(\lam)$ using Algorithm~\ref{alg1}. Second, solve the GEP $\L_{\sigma}(\lam) u =0$ by standard algorithm such as the QZ-algorithm. On the other hand, the decoupling zeros of the LTI system can be obtained by computing eigenvalues of the pencils in (\ref{dcouple}) using standard algorithm such as GUPTRI~\cite{demkog1,demkog2}. As for the transmission zeros of the LTI system, it is well known~\cite{rosenbrock70}  that they coincide with the invariant zeros when the LTI system is controllable and observable, see next section. Thus we have a direct method for  solving {\sc Problem-I}.

\section{Linearizations for rational matrix functions}
We now introduce ``linearization" of a rational matrix function for solving a rational eigenvalue problem.  Intuitively, a linearization of a rational matrix function $G(\lam)$ should be  a matrix pencil $L(\lam)$ of appropriate dimension (ideally the smallest dimension) such that $G(\lam)$ and $L(\lam)$ have identical ``zero structure" in $\C,$ that is, they have the same spectrum  and that  the  partial multiplicities of the zeros (eigenvalues and eigenpoles) of $G(\lam)$ are the same as those of $L(\lam)$. In other words, $\sp(G) = \sp(L)$ and $ \ind_{\phi}(\lam, L) = \ind_{\phi}(\lam, G)$ for each $\lam \in \sp(L).$  Additionally, a linearization $L(\lam)$ should allow an easy recovery of eigenvectors of $G(\lam)$ from those of $L(\lam).$ To achieve this goal, we proceed as follows.

 Let $G(\lam)\in \C(\lam)^{m\times n}$ be a rational matrix function. Then $G(\lam)$ is said to be {\bf proper} if $G(\lam) \rar G_0 \in \C^{m\times n}$ as $\lam \rar \infty.$ On the other hand, $G(\lam)$ is said to be  {\bf  strictly proper} if $G(\lam) \rar 0$ as $\lam \rar \infty$, see~\cite{rosenbrock70, vardulakis}. A rational matrix function $G(\lam)$ can be uniquely decomposed as \be\label{ratdcmp} G(\lam) = P(\lam) + Q(\lam),\ee where $P(\lam)$ is a matrix polynomial and $Q(\lam)$ is a strictly proper rational matrix function. We denote the degree of $P(\lam),$ the highest power of $\lam$ in $P(\lam),$ by $\deg(P).$  Recall that $\mathbf{SM}(G(\lam))$ denotes the Smith-McMillan form of a rational matrix function $G(\lam)$ and $ \mathbf{SF}(P(\lam))$ denotes the Smith form of a matrix polynomial $P(\lam).$ Also recall that $\psi_G(\lam)$ is the pole polynomial of $G(\lam).$ The degree of the pole polynomial $\psi_G(\lam)$ is called the {\bf least order} of the rational matrix function $G(\lam),$ see~\cite{vardulakis}.

\begin{definition}[Linearization for rational matrix function]
Let $G(\lam) \in \C(\lam)^{n \times n}$ be a rational matrix function such that $ G(\lam) = P(\lam) + Q(\lam),$ where $P(\lam)$ is a matrix polynomial and  $Q(\lam)$ is a strictly proper rational matrix function. Suppose that  $$\mathbf{SM}(G(\lam)) = \diag \left(\frac{\phi_{1}(\lam)}{\psi_{1}(\lam)}, \frac{\phi_{2}(\lam)}{\psi_{2}(\lam)}, \ldots, \frac{\phi_{k}(\lam)}{\psi_{k}(\lam)}, 0, \ldots, 0\right), $$ where $k= \nrank(G).$
Set $m := \max(\deg(P), 1)$ and $r := \deg (\psi_{G}),$ where $ \psi_{G}(\lam)$ is the pole polynomial of $G(\lam)$. Then an $(mn+r)\times (mn+r)$ matrix pencil $\mathbb{L}(\lam) = \lam\mathcal{X} + \mathcal{Y} $ is said to be a linearization of $G(\lam)$ if the Smith form of $\mathbb{L}(\lam)$ is given by
$$\mathbf{SF}(\mathbb{L}(\lam)) = \diag\left(I_{(m-1)n+r}, \phi_1(\lam), \cdots, \phi_k(\lam), 0, \ldots, 0\right).$$
\end{definition}

We show that a rational matrix function admits a linearization and that such a linearization can be constructed from a minimal state-space realization of the rational matrix function. A state-space realization of a rational matrix function $G(\lam)$ is an LTI system in SSF  for which $G(\lam)$ is the transfer function~\cite{rosenbrock70, vardulakis, APMA06,  kailath}.

\begin{definition}[State-space realization] A tuple $\Sigma := \left(  \lam E-A,\, B,\, C,\, P(\lam) \right)\in \C[\lam]^{r\times r} \times \C^{r\times n} \times \C^{m\times r} \times \C[\lam]^{m\times n}$ with $ E $ being nonsingular is said to be a {\bf realization} of a rational matrix function $G(\lam) \in \C(\lam)^{m\times n}$ if  $G(\lam) = P(\lam) + C(\lam E - A)^{-1}B.$ In such a case, $G(\lam)$ is said to be the {\bf transfer function} of the realization $\Sigma.$ The matrix polynomial $ \mathcal{S}(\lam) := \left[
                       \begin{array}{c|c}
                        P(\lam) & C \\
                        \hline
                         B & (A - \lam E) \\
                          \end{array}
                          \right]$ is said to be the {\bf Rosenbrock system polynomial} associated with the realization $\Sigma.$ The LTI system in SSF associated with the realization $\Sigma,$ which we also denote by $\Sigma,$ is given by
$$ 
\begin{array}{ll} \Sigma:  &
\begin{array}{ll}
 E \dot{x}(t) = A x(t) + Bu(t) \\
 y(t) = C x(t) + P(\frac{d}{dt}) u(t)
\end{array}
\end{array}
$$ for which $G(\lam)$ is the transfer function and $\mathcal{S}(\lam)$ is the Rosenbrock system polynomial.
 \end{definition}

We mention that the nonsingular matrix $E$ is can be chosen to be the identity matrix $I_r.$  Since  $G(\lam) = P(\lam)+Q(\lam)$ and  $Q(\lam)$ is a strictly proper, there exist ${r\times r}$ matrices  $A$ and $ E$ with $E$ being nonsingular,  $B \in \C^{r\times n}$ and $ C \in \C^{m\times r}$ such that $ Q(\lam) = C(\lam E- A)^{-1}B,$ see~\cite{rosenbrock70}.  Thus  $ G(\lam) = P(\lam) + C(\lam E - A)^{-1}B.$ Obviously, a state-space realization of $G(\lam)$ is not unique. We refer to~\cite{APMA06, kailath, rosenbrock70, vardulakis} for theory and algorithms on realizations of rational matrix functions.

A realization $\left(\lam E-A,\, B,\, C,\, P(\lam) \right)$ of $G(\lam)$ is said to be {\bf minimal} if the size of the pencil $\lam E- A$ is the smallest~\cite{rosenbrock70, vardulakis}. The minimality of a realization is characterized by controllability and observability  of the associated LTI system in SSF.

\begin{theorem}\cite{rosenbrock70, vardulakis}\label{minreal} Let  $\Sigma  := \left(  \lam E-A,\, B,\, C,\, P(\lam) \right)\in \C[\lam]^{r\times r} \times \C^{r\times n} \times \C^{m\times r} \times \C[\lam]^{m\times n}$ be a realization of $G(\lam).$   Then the following conditions are equivalent.
\begin{itemize}
\item[(a)] The realization  $\Sigma $ is minimal.

\item[(b)] The  LTI system associated with the realization $\Sigma$ is controllable and observable. Equivalently,  $$\rank\left(\left[\begin{matrix} \lam E -A & B\end{matrix}\right]\right) = r  \mbox{ and } \rank\left(\left[\begin{matrix} \lam E -A \\ C\end{matrix}\right]\right) = r \mbox{ for all } \lam \in \C.$$

\item[(c)] The size of $\lam E- A$ is equal to the least order of $G(\lam),$ that is, $ r = \deg(\psi_G),$ where $\psi_G(\lam)$ is the pole polynomial of $G(\lam).$
\end{itemize}
\end{theorem}

A  realization of $G(\lam)$ which is not minimal can always be reduced to a minimal realization~\cite{rosenbrock70, kailath}.  Observe that if $\left(  \lam E-A,\, B,\, C,\, P(\lam) \right)$ is a minimal realization of $G(\lam)$ then the associated LTI system $\Sigma$ does not have input decoupling zeros as well as output decoupling zeros. In such a case,  $\lam \in \C$ is a pole of $G(\lam)$ if and only if $\lam$ is an  eigenvalue of the pencil $\lam E -A.$ Further, the zeros of $G(\lam)$ coincide with the eigenvalues of the Rosenbrock system polynomial $\mathcal{S}(\lam)$ associated with the LTI system $\Sigma.$  In other words, the transmission zeros of the  LTI system $\Sigma$ coincide with the invariant zeros of the system. This can be seen as follows.

\begin{theorem}\cite{rosenbrock70}\label{pmd} Let $ \mathbb{P}(\lam)$ be an $ (r+m)\times (r+n)$ matrix polynomial given by
  $$\mathbb{P}(\lam) := \left[
                            \begin{array}{c|c}
                              - T(\lam) & U(\lam) \\
                              \hline
                              V(\lam) & W(\lam) \\
                            \end{array}
                          \right],
$$ where $T(\lam) \in \C[\lam]^{r\times r}$ is regular and  $\rank \left(\left[\begin{matrix} -T(\lam) & U(\lam) \end{matrix}\right] \right) =  \rank \left(\left[\begin{matrix} -T(\lam) \\ V(\lam) \end{matrix}\right] \right) = r$ for all $ \lam \in \C.$ Consider  the $m\times n$ rational matrix function $G(\lam)$  given by $$ G(\lam) := W(\lam) + V(\lam) (T(\lam))^{-1}U(\lam).$$ Suppose that the Smith-McMillan form of $G(\lam)$  is given by
$$\mathbf{SM}(G(\lam)) = \diag \left(\phi_{1}(\lam)/ \psi_{1}(\lam), \ldots, \phi_{q}(\lam)/ \psi_{q}(\lam), 0_{m-q, n-q}\right). $$ Then the Smith forms of $T(\lam)$ and  $\mathbb{P}(\lam)$ are given by
\beano
 \mathbf{SF}(T(\lam)) &=& \diag\left(I_{r-q}, \psi_{q}(\lam), \psi_{q-1}(\lam)\ldots, \psi_{1}(\lam)\right), \\
\mathbf{SF}(\mathbb{P}(\lam)) &=& \diag \left(I_{r}, \phi_{1}(\lam), \phi_{2}(\lam)\ldots, \phi_{q}(\lam), 0_{m-q, n-q}\right).
\eeano
\end{theorem}

For a minimal realization of a rational matrix function we have the following.

\begin{theorem} \label{zpspespr} Let $\left(  \lam E-A,\, B,\, C,\, P(\lam) \right)\in \C[\lam]^{r\times r} \times \C^{r\times n} \times \C^{m\times r} \times \C[\lam]^{m\times n}$ be a minimal realization of $G(\lam) \in \C(\lam)^{m\times n}$ and $\mathcal{S}(\lam)$  be the associated Rosenbrock system polynomial. Consider the pencil $L(\lam) := \lam E - A$ and suppose that $$\mathbf{SM}(G(\lam)) = \diag \left(\frac{\phi_{1}(\lam)}{\psi_{1}(\lam)}, \frac{\phi_{2}(\lam)}{\psi_{2}(\lam)}, \ldots, \frac{\phi_{k}(\lam)}{\psi_{k}(\lam)}, 0, \ldots, 0\right). $$
\begin{itemize}

\item [(a)] Then the Smith form $\mathbf{SF}(L(\lam))$ of $L(\lam)$ is given by $$\mathbf{SF}(L(\lam)) = \diag \left(I_{r-k}, \psi_{k}, \psi_{k-1}, \ldots, \psi_{1} \right)$$

\item [] and the Smith form $\mathbf{SF}(\mathcal{S}(\lam))$ of $\mathcal{S}(\lam)$ is given by$$\mathbf{SF}(\mathcal{S}(\lam)) = \diag \left(I_{r}, \phi_{1}, \phi_{2}, \ldots, \phi_{k}, 0, \ldots, 0 \right). $$

\item [(b)] We have  $\poles(G) = \sp(L)$ and  $\ind_{\psi}(\lam, G) = \ind_{\phi}(\lam, L)$ for $\lam \in \sp(L). $

\item [(c)] We have  $\sp(G) = \sp(\mathcal{S})$ and $\ind_{\phi}(\lam, G) = \ind_{\phi}(\lam, \mathcal{S})$ for $\lam \in \sp(\mathcal{S}).$

\item [(d)] We have  $\mathrm{Eip}(G) = \sp(\mathcal{S}) \cap \sp(L) $ and $\mathrm{Eig}(G) = \sp(\mathcal{S}) \cap ( \C\setminus \sp(L)).$



\end{itemize}

\end{theorem}
\begin{proof}
Since the realization is minimal, the results in  part(a) follow from Theorem~\ref{minreal} and Theorem~\ref{pmd}. On the other hand, the results in parts (b), (c) and (d) follow from part (a) and Theorem~\ref{eigpole}.
\end{proof}

We now show that Fiedler pencils of the Rosenbrock system polynomial associated with a minimal realization of $G(\lam)$ are in fact  linearizations of $G(\lam).$

\begin{theorem}[Linearization] Let $\left(  \lam E-A,\, B,\, C,\, P(\lam) \right)\in \C[\lam]^{r\times r} \times \C^{r\times n} \times \C^{n\times r} \times \C[\lam]^{n\times n}$ be a minimal realization of $G(\lam) \in \C(\lam)^{n\times n}$ and $\mathcal{S}(\lam)$  be the associated Rosenbrock system polynomial. Suppose that $m$ is the degree of $P(\lam).$   Let $\mathbb{L}_{\sig}(\lam)$ be the Fiedler pencil of  $\mathcal{S}(\lam)$ associated with a bijection $\sigma : \{0, 1, \ldots, m-1\} \rar \{1, 2, \ldots, m\}.$  Then $\L_{\sigma}(\lam)$ is a linearization of the rational matrix function $G(\lam).$
\end{theorem}

\begin{proof}
Suppose that the Smith form of $G(\lam)$ is given by $$\mathbf{SM}(G(\lam)) = \diag \left(\frac{\phi_{1}(\lam)}{\psi_{1}(\lam)}, \frac{\phi_{2}(\lam)}{\psi_{2}(\lam)}, \ldots, \frac{\phi_{k}(\lam)}{\psi_{k}(\lam)}, 0, \ldots, 0\right). $$  By Theorem \ref{zpspespr}, the Smith form $\mathbf{SF}(\mathcal{S}(\lam))$ of $\mathcal{S}(\lam)$ is given by $$\mathbf{SF}(\mathcal{S}(\lam)) = \diag \left(I_{r}, \phi_{1}, \phi_{2}, \ldots, \phi_{k}, 0, \ldots, 0 \right). $$
Since $\mathbb{L}_{\sig}(\lam)$ is a Rosenbrock linearization of $\mathcal{S}(\lam),$  there are  $nm \times nm$ unimodular matrix polynomials $U(\lam)$ and $V(\lam)$ such that $$\left[
         \begin{array}{c|c}
           U(\lam) & 0 \\
           \hline
           0 & I_r \\
         \end{array}
       \right] \mathbb{L}_{\sigma}(\lam) \left[
                                    \begin{array}{c|c}
                                     V(\lam) & 0 \\
                                     \hline
                                     0 & I_r \\
                                \end{array}
                               \right] = \left[
                                           \begin{array}{c|c}
                                             I_{(m-1)n} & 0 \\
                                             \hline
                                             0 &  \mathcal{S}(\lam) \\
                                           \end{array}
                                         \right].$$
 As $\mathcal{S}(\lam)$ is unimodularly equivalent to $\mathbf{SF}(\mathcal{S}(\lam)),$ it follows that $\L_{\sigma}(\lam)$ is unimodularly equivalent to $\diag \left(I_{(m-1)n+r}, \phi_{1}, \phi_{2}, \ldots, \phi_{k}, 0, \ldots, 0 \right).$ Hence the Smith form $\mathbf{SF}(\mathcal{\L_{\sigma}}(\lam))$ is given by
 $$\mathbf{SF}(\mathbb{L}_{\sig}(\lam)) = \diag \left(I_{(m-1)n+r}, \phi_{1}, \phi_{2}, \ldots, \phi_{k}, 0, \ldots, 0 \right) $$ proving that   $\mathbb{L}_{\sig}(\lam)$ is a linearization of $G(\lam)$.
\end{proof}

Thus we conclude that the zeros of a rational matrix function $G(\lam) \in \C(\lam)^{n\times n}$ in general and the eigenvalues of $G(\lam)$ in particular can be computed by solving the generalized eigenvalue problem for a Fiedler pencil of the Rosenbrock system polynomial associated with a minimal realization of $G(\lam).$ Hence a direct method for solving a rational eigenvalue problem $G(\lam) u = 0$ can be summed up as follows.

\begin{algorithm}[H]
\caption{ A direct method for solving an REP  $G(\lam) u =0.$ }
\label{alg2}
\begin{enumerate}
\item Compute a minimal realization $\Sigma := \left(  \lam E-A,\, B,\, C,\, P(\lam) \right)$ of $G(\lam).$
\item Compute a Fiedler pencil $\L_{\sigma}(\lam) $ of the Rosenbrock system polynomial associated with $\Sigma$ using Algorithm~\ref{alg1}.
\item Solve the generalized eigenvalue problem $ \L_{\sigma}(\lam) v = 0.$
\end{enumerate}

\end{algorithm}

Note that the eigenvalues of $\L_{\sigma}(\lam)$ are precisely the eigenvalues and eigenpoles of $G(\lam),$ that is, $\sp(G) = \sp(\L_{\sigma}).$  Also note that the eigenpoles of $G(\lam)$ are those eigenvalues of $\L_{\sigma}(\lam)$ which are also eigenvalues of the pencil $L(\lam) := A - \lam E,$ that is, $ \mathrm{Eip}(G) = \sp(\L_{\sigma})\cap \sp(L).$  We mention that the pencil $\L_{\sigma}(\lam)$ allows easy recovery of eigenvectors of $G(\lam)$ from those of $\L_{\sigma}(\lam)$ which is considered in a forthcoming paper.

\section{Conclusions and future work} Computation of zeros of a linear time-invariant (LTI) system is a major task in Linear Systems Theory. For computing zeros of an LTI system $\Sigma$ in SSF, we have introduced Rosenbrock linearizations, which are trimmed structured linearizations,  of the Rosenbrock system polynomial $\mathcal{S}(\lam)$ associated with the LTI system $\Sigma.$ We have introduced Fiedler-like pencils for the Rosenbrock system polynomial $\mathcal{S}(\lam)$ and have shown that the Fiedler-like pencils are in fact Rosenbrock linearizations of $\mathcal{S}(\lam).$ For solving rational eigenvalue problems, we have introduced linearizations of a rational matrix function by utilizing a minimal state-space realization of the rational matrix function. We have introduced Fiedler-like pencils for a rational matrix function  by embedding the rational eigenproblem into an LTI system in SSF which is controllable as well as observable and have shown that the Fiedler-like pencils are in fact linearizations of the rational matrix function. Other important issues such as the eigenvector recovery of a rational matrix function from that of its linearizations, structured preserving Fiedler-like pencils (such as generalized Fiedler pencils and generalized Fiedler pencils with repetition) for Rosenbrock system polynomials and rational matrix functions will be considered in forthcoming papers.

%
%

\begin{thebibliography}{10}



\bibitem{AV}{\sc E. N. Antoniou and  S. Vologiannidis,}  {\em A new family of companion forms of polynomial matrices},
 {Electron. J. Linear Algebra, 11(2004), pp.78-87.}


\bibitem{APMA06}{\sc P. J. Antsaklis and A. N. Michel,} {\em Linear systems,} { Birkh\"auser Boston Inc., Boston, MA,  2006. }


\bibitem{planchard89}{\sc C. Conca, J. Planchard and M. Vanninathan,} {\em Existence and location of eigenvalues for fluid-solid structures,}
{ Comput. Methods Appl. Mech. Engrg., 77(1989), pp.253-291.



\bibitem{demkog1}  {\sc J. Demmel and B. Kagstrom, } {\em     The generalized Schur decomposition of an arbitrary pencil ${A} - \lambda {B}$: Robust software with error bounds and applications. Part I: Theory and algorithms,} {ACM Trans. Math. Software, 19(1993), pp. 160-174.}

\bibitem{demkog2} {\sc  J. Demmel and B. Kagstrom,}  {\em  The generalized Schur decomposition of an arbitrary pencil ${A} - \lambda {B}$: Robust software with error bounds and applications. Part II: Software and applications,} { ACM Trans. Math. Software, 19(1993), pp.175-201.}

\bibitem{TDM}{\sc F. De Ter{\'a}n, F.  Dopico, and D. S. Mackey}, {\em Fiedler companion linearizations and the recovery of minimal indices,}
{ SIAM J. Matrix Anal. Appl.,31(2009/10), pp. 2181-2204.}


\bibitem{gohberg82} {\sc I. Gohberg,  P. Lancaster and L. Rodman,}
{\em Matrix polynomials,}{Academic Press Inc., New York  London, 1982.}

%
%

\bibitem{hwang04}{\sc Tsung-Min Hwang,  Wen-Wei Lin,  Wei-Cheng  Wang,  and
Weichung  Wang, } {\em Numerical simulation of three dimensional pyramid quantum dot,}
{ J. Comput. Phys., 196 (2004), pp.208-232.}



\bibitem{voss05}{\sc E. Jarlebring and H. Voss,}
{\em Rational Krylov for nonlinear eigenproblems: an iterative
projection method.} { Appl. Math.,50(2005), pp.543-554.}


\bibitem{kailath} {\sc T. Kailath,} {\em Linear systems,} {Prentice-Hall Inc., Englewood Cliffs, N.J., 1980.}

%


\bibitem{mackey2006vs}{\sc D. S.  Mackey, N. Mackey, C. Mehl and V. Mehrmann,}
{\em Vector spaces of linearizations for matrix polynomials,} { SIAM J. Matrix Anal. Appl., 28(2006), pp.971-1004.}

%
%

\bibitem{volker2004}{\sc V. Mehrmann and H. Voss,} {\em Nonlinear eigenvalue problems: a challenge for modern
eigenvalue methods,}  GAMM Mitt. Ges. Angew. Math. Mech.,  27(2004), pp.121-152.}


\bibitem{planchard82}{\sc J. Planchard,} {\em Eigenfrequencies of a tube bundle placed in a confined fluid,}
{ Comput. Methods Appl. Mech. Engrg., 30(1982), pp.75-93.}


\bibitem{rosenbrock70} {\sc H. H. Rosenbrock,} {\em State-space and multivariable theory,}
{ John Wiley \& Sons, Inc., New York, 1970.}


\bibitem{ruhe73}{\sc A. Ruhe,} {\em Algorithms for the nonlinear eigenvalue problem,} { SIAM J. Numer. Anal.,10(1973), pp.674-689.}


\bibitem{solo}{\sc S. I. Solov{\"e}v, } {\em Preconditioned iterative methods for a class of nonlinear
eigenvalue problems,} { Linear Algebra Appl., 415(2006), pp.210-229.}


\bibitem{bai11}{\sc Y. Su and Z. Bai,} {\em Solving rational eigenvalue problems via linearization,}
{ SIAM J. Matrix Anal. Appl., 32(2011), pp.201-216.}


\bibitem{vmct06}{\sc T. Betcke, N. J. Higham, V. Mehrmann, C. Schr{\"o}der,  and F. Tisseur,}
{\em NLEVP: a collection of nonlinear eigenvalue problems,} { ACM Trans. Math. Software,39(2013), pp.7-28.}


\bibitem{vardulakis}{\sc A. I. G. Vardulakis,} {\em Linear multivariable control,} { John Wiley \& Sons Ltd., 1991.}


\bibitem{hvoss04}{\sc H. Voss, } {\em An Arnoldi method for nonlinear eigenvalue problems,} { BIT, 44(2004), pp.387-401.}


\bibitem{voss82}{\sc H. Voss  and B.  Werner, } {\em A minimax principle for nonlinear eigenvalue problems with
applications to nonoverdamped systems,} { Math. Methods Appl. Sci., 4(1982), pp.415-424.}


\bibitem{voss1}{\sc H. Voss, } {\em A rational spectral problem in fluid-solid vibration,} { Electron. Trans. Numer. Anal.,  16(2003), pp.93-105.}


\bibitem{voss04}{\sc H. Voss,} {\em A Jacobi-Davidson method for nonlinear eigenproblems,} { Lecture Notes Comp. Sci., 3037 (2004), pp. 34-41.}


\bibitem{vossh06}{\sc H. Voss, } {\em Iterative projection methods for computing relevant energy states of a quantum dot,} { J. Comput. Phys.,  217(2006), pp.824-833.}



\end{thebibliography}

%
%
%

%
%
%

\end{document}